\documentclass[secthm,seceqn,amsthm,ussrhead,reqno, 12pt]{amsart}
\usepackage[utf8]{inputenc}
\usepackage[english]{babel}
\usepackage{amssymb,amsmath,amsthm,amsfonts,xcolor,enumerate,hyperref,comment,longtable,cleveref}

\usepackage{cite}
\usepackage{amsaddr}
\usepackage{xcolor,mathrsfs}
\usepackage{pgf,tikz}
\usetikzlibrary{arrows}

\allowdisplaybreaks

\newtheorem{thrm}{Theorem}[section]
\newtheorem{cor}[thrm]{Corollary}
\newtheorem{lem}[thrm]{Lemma}
\newtheorem{prop}[thrm]{Proposition}

\theoremstyle{definition}
\newtheorem{defn}[thrm]{Definition}

\crefrangeformat{equation}{#3(#1)#4--#5(#2)#6}

\crefname{thrm}{Theorem}{Theorems}
\crefname{lem}{Lemma}{Lemmas}
\crefname{cor}{Corollary}{Corollaries}
\crefname{prop}{Proposition}{Propositions}
\crefname{defn}{Definition}{Definitions}
\crefname{exm}{Example}{Examples}
\crefname{rem}{Remark}{Remarks}
\crefname{section}{Section}{Sections}
\crefname{equation}{\unskip}{\unskip}
\crefname{enumi}{\unskip}{\unskip}

\DeclareMathOperator{\Ann}{Ann}

\newcommand{\af}{\alpha}
\newcommand{\bt}{\beta}
\newcommand{\gm}{\gamma}
\newcommand{\dl}{\delta}
\newcommand{\ve}{\varepsilon}
\newcommand{\Dl}{\Delta}

\newcommand{\vf}{\varphi}

\newcommand{\B}{\mathcal{B}}
\newcommand{\C}{\mathbb{C}}
\newcommand{\Z}{\mathbb{Z}}
\renewcommand{\S}{\mathcal{S}}

\newcommand{\sst}{\subseteq}

\newcommand{\id}{\mathrm{id}}

\newcommand{\gen}[1]{\langle #1 \rangle}

\begin{document}

	\noindent{\Large  
		Transposed Poisson structures on Block Lie algebras and superalgebras}\footnote{
		The first  author is supported by RSF  22-11-00081;  
		the second author  is supported by  CNPq     404649/2018-1 and 	302980/2019-9, 
		FCT   UIDB/MAT/00212/2020 and UIDP/MAT/00212/2020.}

	\
	
	{\bf
		Ivan Kaygorodov$^{a,b,c}$ \& Mykola Khrypchenko$^{d}$  \\

		\medskip
	}
	
	{\tiny

		$^{a}$ Centro de Matemática e Aplicações, Universidade da Beira Interior,  Portugal
		
		\smallskip

$^{b}$ Saint Petersburg  University, Russia
		\smallskip

$^{c}$ Moscow Center for Fundamental and Applied Mathematics,      Russia
		\smallskip

		$^{d}$ Departamento de Matem\'atica, Universidade Federal de Santa Catarina,     Brazil

		\
		
		\smallskip
		
		\medskip
		
		E-mail addresses:

		\smallskip

		Ivan Kaygorodov (kaygorodov.ivan@gmail.com)

		\smallskip
		
		Mykola Khrypchenko  
		(nskhripchenko@gmail.com)
		
	}

	\medskip

	\ 
	
	\noindent {\bf Abstract:} {\it 	
We describe transposed Poisson algebra structures on  Block Lie algebras $\B(q)$ and Block Lie superalgebras $\S(q)$, where $q$ is an arbitrary complex number. Specifically, we show that the transposed Poisson structures on $\B(q)$ are trivial whenever $q\not\in\Z$, and for each $q\in\Z$ there is only one (up to an isomorphism) non-trivial transposed Poisson structure on $\B(q)$. The superalgebra $\S(q)$ admits only trivial transposed Poisson superalgebra structures for $q\ne 0$ and two non-isomorphic non-trivial transposed Poisson superalgebra structures for $q=0$. As a consequence, new Lie algebras and superalgebras that admit non-trivial ${\rm Hom}$-Lie algebra structures are found. 
}

	\
	

	
	\tableofcontents
	
	\section*{Introduction}\label{intro}
The origin of Poisson algebras lies in the Poisson geometry of the 1970s, and since then these algebras have shown their importance in several areas of mathematics and physics, such as Poisson manifolds, algebraic geometry, operads, quantization theory, quantum groups, and classical and quantum mechanics. The study of all possible Poisson algebra structures with fixed Lie or associative part is a popular topic in the theory of Poisson algebras~\cite{jawo,said2,YYZ07,kk21}.
Recently, Bai, Bai, Guo, and Wu~\cite{bai20} have introduced a dual notion of the Poisson algebra, called \textit{transposed Poisson algebra}, by exchanging the roles of the two binary operations in the Leibniz rule defining the Poisson algebra. 
They have shown that a transposed Poisson algebra defined this way not only shares common properties of a Poisson algebra, including the closedness under tensor products and the Koszul self-duality as an operad, but also admits a rich class of identities. More significantly, a transposed Poisson algebra naturally arises from a Novikov-Poisson algebra by taking the commutator Lie algebra of the Novikov algebra.
Thanks to \cite{bfk22}, 
any unital transposed Poisson algebra is
a particular case of a ``contact bracket'' algebra~\cite{zel} 
and a quasi-Poisson algebra~\cite{billig}.
Later, in a recent paper by Ferreira, Kaygorodov, and  Lopatkin
a relation between $\frac{1}{2}$-derivations of Lie algebras and 
transposed Poisson algebras have been established \cite{FKL}.
These ideas were used to describe all transposed Poisson structures 
on  Witt and Virasoro algebras in  \cite{FKL};
on   twisted Heisenberg-Virasoro,   Schrodinger-Virasoro  and  
  extended Schrodinger-Virasoro algebras in \cite{yh21};
on   oscillator algebras in  \cite{bfk22}.
Fehlberg Júnior and Kaygorodov gave a way to construct new transposed Poisson algebras by the Kantor product of their multiplications in \cite{FK21}.
The ${\rm Hom}$- and ${\rm BiHom}$-versions of transposed Poisson algebras have been considered in \cite{hom, bihom}.	 
A list of actual open questions on transposed Poisson algebras is given in \cite{bfk22}.

Block Lie algebras is a class of simple infinite dimensional Lie algebras introduced by Block in 1958 \cite{block58}.
Since then, several generalizations of these algebras have appeared \cite{oz99, xu99,dz96}. Block Lie algebras, their generalizations and related algebras are still under an active investigation \cite{cgz14,suyue15,suxia18,xia19,suxia20,suxixu12}.
Thus, Đoković and Zhao introduced a generalization of Block algebras and  described their derivations, isomorphisms and second cohomology in \cite{dz96}. 
Xia,  You and  Zhou~\cite{xyz12} 
defined, for a fixed complex number $q$, the Block   algebra $\B(q)$ as a Lie algebra with a basis $\{L_{m,i} \mid m, i \in \Z\}$ and the following multiplication table 
$$
 [L_{m,i}, L_{n,j}] = (n(i + q) - m(j + q))L_{m+n,i+j},\ i, j, m, n \in \Z.
$$
They proved that, for distinct integers $q_1$ and $q_2$, the algebras $\B(q_1)$ and $\B(q_2)$ are non-isomorphic, and calculated the automorphism group and the algebra of derivations of $\B(q)$ for an arbitrary $q\in\C.$ 
They also showed that the second scalar cohomology group of $\B(q)$ is one-dimensional and found the unique non-trivial central extension of $\B(q)$.
Liu, Guo, Xiangqian and  Zhao~\cite{lgz18} described all biderivations
(and, as an application, all commuting maps) on Block algebras $\B(q)$.
$U(h)$-free modules over the Block algebra $\B(q)$ have been studied by Guo, Wang and Liu  in \cite{gwl21}.
 Xia,  You and  Zhou introduced a superanalog of Block Lie algebras in \cite{xia16}.

In the present paper, we give a full description of transposed Poisson algebra structures on
Block  Lie algebras $\B(q)$ defined in \cite{xyz12} and transposed Poisson superalgebra structures on Block Lie superalgebras $\S(q)$ defined in \cite{xia16}. More precisely, we prove in \cref{thalg} that transposed Poisson algebra structures on $\B(q)$ are  trivial for $q\not\in\Z$, and there is only non-trivial such structure for $q\in \Z$. We also show in \cref{tp-on-S(q)} that $\S(q)$ admits only trivial transposed Poisson superalgebra structures for $q\ne 0$, but in the case $q=0$ there are two non-trivial such structures.

	\section{Preliminaries}\label{prelim}
	
	\subsection{Transposed Poisson algebras}

All the algebras below will be over $\mathbb C$ and all the linear maps will be $\mathbb C$-linear, unless otherwise stated.

\begin{defn}\label{tpa}
Let ${\mathfrak L}$ be a vector space equipped with two nonzero bilinear operations $\cdot$ and $[\cdot,\cdot].$
The triple $({\mathfrak L},\cdot,[\cdot,\cdot])$ is called a \textit{transposed Poisson algebra} if $({\mathfrak L},\cdot)$ is a commutative associative algebra and
$({\mathfrak L},[\cdot,\cdot])$ is a Lie algebra that satisfies the following compatibility condition
\begin{align}\label{Trans-Leibniz}
2z\cdot [x,y]=[z\cdot x,y]+[x,z\cdot y].
\end{align}
\end{defn}

Transposed Poisson algebras were first introduced in a paper by Bai, Bai, Guo, and Wu \cite{bai20}.

\begin{defn}\label{tp-structures}
    Let $({\mathfrak L},[\cdot,\cdot])$ be a Lie algebra. A \textit{transposed Poisson algebra structure} on $({\mathfrak L},[\cdot,\cdot])$ is a commutative associative multiplication $\cdot$ on $\mathfrak L$ which makes $({\mathfrak L},\cdot,[\cdot,\cdot])$ a transposed Poisson algebra.
\end{defn}

\begin{defn}\label{12der}
Let $({\mathfrak L}, [\cdot,\cdot])$ be an algebra and $\varphi:\mathfrak L\to\mathfrak L$ a linear map.
Then $\varphi$ is a \textit{$\frac{1}{2}$-derivation} if it satisfies
\begin{align}\label{vf(xy)=half(vf(x)y+xvf(y))}
\varphi([x,y])= \frac{1}{2} \left([\varphi(x),y]+ [x, \varphi(y)] \right).
\end{align}
\end{defn}
Observe that $\frac{1}{2}$-derivations are a particular case of $\delta$-derivations introduced by Filippov in \cite{fil1}
(see also \cite{k12,z10} and references therein).

\cref{tpa,12der} immediately imply the following key Lemma.
\begin{lem}\label{glavlem}
Let $({\mathfrak L},\cdot,[\cdot,\cdot])$ be a transposed Poisson algebra 
and $z\in{\mathfrak L}.$
Then the left multiplication $L_z$ of $({\mathfrak L},\cdot)$ is a $\frac{1}{2}$-derivation of $({\mathfrak L}, [\cdot,\cdot]).$
\end{lem}

The basic example of a $\frac{1}{2}$-derivation is the multiplication by a field element.
Such $\frac{1}{2}$-derivations will be called \textit{trivial}.

\begin{thrm}\label{princth}
Let ${\mathfrak L}$ be a Lie algebra without non-trivial $\frac{1}{2}$-derivations.
Then all transposed Poisson algebra structures on ${\mathfrak L}$ are trivial.
\end{thrm}

Let us recall the definition of ${\rm Hom}$-structures on Lie algebras.
\begin{defn}
Let $({\mathfrak L}, [\cdot,\cdot])$ be a Lie algebra and $\varphi$ be a linear map.
Then $({\mathfrak L}, [\cdot,\cdot], \varphi)$ is a ${\rm Hom}$-Lie structure on $({\mathfrak L}, [\cdot,\cdot])$ if 
\[
[\varphi(x),[y,z]]+[\varphi(y),[z,y]]+[\varphi(z),[x,y]]=0.
\]
\end{defn}

\subsection{Transposed Poisson superalgebras}
One naturally defines a transposed Poisson superalgebra as a superization of the notion of a transposed Poisson algebra.

\begin{defn}
Let ${\mathfrak L}={\mathfrak L}_0 \oplus {\mathfrak L}_1$ be a $\mathbb{Z}_2$-graded vector space 
equipped with two nonzero bilinear super-operations $\cdot$ and $[\cdot,\cdot].$
The triple $({\mathfrak L},\cdot,[\cdot,\cdot])$ is called a \textit{transposed Poisson superalgebra} if 
$({\mathfrak L},\cdot)$ is a supercommutative associative superalgebra and
$({\mathfrak L},[\cdot,\cdot])$ is a Lie superalgebra that satisfies the following compatibility condition
\begin{align}\label{super-trans-leibniz}
2z\cdot [x,y]=[z\cdot x,y]+ (-1)^{|x||z|}[x,z\cdot y], \ x,y,z \in {\mathfrak L}_0 \cup {\mathfrak L}_1.
\end{align}
\end{defn}
	
\begin{defn}\label{tp-superstructures}
    Let $({\mathfrak L},[\cdot,\cdot])$ be a Lie superalgebra. A \textit{transposed Poisson superalgebra structure} on $({\mathfrak L},[\cdot,\cdot])$ is a supercommutative associative multiplication $\cdot$ on $\mathfrak L$ which makes $({\mathfrak L},\cdot,[\cdot,\cdot])$ a transposed Poisson superalgebra.
\end{defn}

\begin{defn}\label{12superder}
Let $({\mathfrak L}, [\cdot,\cdot])$ be a superalgebra and $\varphi$ a homogeneous linear map $\mathfrak L\to\mathfrak L$.
Then $\varphi$ is called a \textit{$\frac{1}{2}$-superderivation} if it satisfies
\begin{align*}
\varphi([x,y])= \frac{1}{2} \left([\varphi(x),y]+ (-1)^{|\varphi| |x|} [x, \varphi(y)] \right), \ 
 \ x,y \in {\mathfrak L}_0 \cup {\mathfrak L}_1.
\end{align*}
\end{defn}

\begin{lem}\label{glavsuperlem}
Let $({\mathfrak L},\cdot,[\cdot,\cdot])$ be a transposed Poisson superalgebra and $z\in{\mathfrak L}_0 \cup {\mathfrak L}_1.$
Then the left multiplication $L_z$ of $({\mathfrak L},\cdot)$ is a $\frac{1}{2}$-superderivation of $({\mathfrak L}, [\cdot,\cdot])$ and $|L_z|=|z|$.
\end{lem}


Let $\cdot$ be a transposed Poisson (super)algebra structure on a Lie (super)algebra $({\mathfrak L}, [\cdot,\cdot])$. Then any automorphism $\phi$ of $({\mathfrak L}, [\cdot,\cdot])$ induces another transposed Poisson (super)algebra structure $*$ on $({\mathfrak L}, [\cdot,\cdot])$ given by
\begin{align*}
    x*y=\phi(\phi^{-1}(x)\cdot\phi^{-1}(x)),\ \ x,y\in{\mathfrak L}.
\end{align*}
Clearly, $\phi$ is an isomorphism of transposed Poisson (super)algebras $({\mathfrak L},\cdot,[\cdot,\cdot])$ and $({\mathfrak L},*,[\cdot,\cdot])$. 

	\section{Transposed Poisson structures on Block Lie algebras}

	\begin{defn}
		Let $q$ be a fixed complex number. The \textit{Block Lie algebra} $\B(q)$ is the complex Lie algebra with a basis $\{L_{m,i} \mid m, i \in \Z\}$, where
		\begin{align}\label{[L_mi_L_nj]=(n(i+q)-m(j+q))L_m+n_i+j}
			[L_{m,i}, L_{n,j}] = (n(i + q) - m(j + q))L_{m+n,i+j}
		\end{align}
		for all $i, j, m, n \in \Z$.
	\end{defn}
	
	Observe that $\B(q)=\bigoplus_{m,i\in\Z}\B(q)_{m,i}$ is a $\Z\times\Z$-grading, where $\B(q)_{m,i}=\C L_{m,i}$.
	\subsection{$\frac 12$-derivations of $\B(q)$}
	
	Let $\vf$ be a linear map $\B(q)\to\B(q)$. Then the $\Z\times\Z$-grading of $\B(q)$ induces the decomposition
	\begin{align*}
		\vf=\sum_{r,s\in\Z}\vf_{r,s},
	\end{align*}
	where $\vf_{r,s}$ is a linear map $\B(q)\to\B(q)$ such that $\vf_{r,s}(\B(q)_{m,i})\sst \B(q)_{m+r,i+s}$ for all $m,i\in\Z$. Observe that $\vf$ is a $\frac 12$-derivation of $\B(q)$ if and only if $\vf_{r,s}$ is a $\frac 12$-derivation of $\B(q)$ for all $r,s\in\Z$. We write 
	\begin{align}\label{Dl_rs(L_mi)=d_rs(m_i)L_m+r_i+s}
		\vf_{r,s}(L_{m,i})=d_{r,s}(m,i)L_{m+r,i+s},
	\end{align}
	where $d_{r,s}(m,i)\in\C$.
	
	\begin{lem}\label{conditions-on-d}
		Let $\vf_{r,s}:\B(q)\to\B(q)$ be a linear map satisfying \cref{Dl_rs(L_mi)=d_rs(m_i)L_m+r_i+s}. Then $\vf_{r,s}$ is a $\frac 12$-derivation of $\B(q)$ if and only if 
		\begin{align}\label{one-half-der-in-terms-of-d_rs}
		2(n(i+q)-m(j+q))d_{r,s}(m+n,i+j)&=(n(i+s+q)-(m+r)(j+q))d_{r,s}(m,i)\notag\\
		&\quad+((n+r)(i+q)-m(j+s+q))d_{r,s}(n,j).
		\end{align}
	\end{lem}
	\begin{proof}
		It is obvious that $\vf_{r,s}$ is a $\frac 12$-derivation of $\B(q)$ if and only if \cref{vf(xy)=half(vf(x)y+xvf(y))} holds on the basis $\{L_{m,i} \mid m, i \in \Z\}$. Using \cref{[L_mi_L_nj]=(n(i+q)-m(j+q))L_m+n_i+j,Dl_rs(L_mi)=d_rs(m_i)L_m+r_i+s}, for all $i, j, m, n \in \Z$ we have
		\begin{align*}
		2\vf_{r,s}([L_{m,i},L_{n,j}])&= 2\vf_{r,s}((n(i + q) - m(j + q))L_{m+n,i+j})\\
		&=2(n(i + q) - m(j + q))d_{r,s}(m+n,i+j)L_{m+n+r,i+j+s}.
		\end{align*}
		On the other hand,
		\begin{align*}
		[\vf_{r,s}(L_{m,i}),L_{n,j}]+[L_{m,i},\vf_{r,s}(L_{n,j})]&=[d_{r,s}(m,i)L_{m+r,i+s},L_{n,j}]+[L_{m,i},d_{r,s}(n,j)L_{n+r,j+s}]\\
		&=(n(i+s+q)-(m+r)(j+q))d_{r,s}(m,i)L_{m+n+r,i+j+s}\\
		&\quad+((n+r)(i+q)-m(j+s+q))d_{r,s}(n,j)L_{m+n+r,i+j+s}.
		\end{align*}
		Thus, \cref{vf(xy)=half(vf(x)y+xvf(y))} is equivalent to \cref{one-half-der-in-terms-of-d_rs}.
	\end{proof}

	\subsubsection{The case $q\ne 0$}
	
	\begin{lem}\label{d_rs-is-zero-for-r-ne-0}
		Let $\vf$ be a $\frac 12$-derivation of $\B(q)$ and $r\ne 0$. Then $\vf_{r,s}=0$.
	\end{lem}
	\begin{proof}
		Taking $n=j=0$ in \cref{one-half-der-in-terms-of-d_rs}, we obtain
		\begin{align*}
			-2mq\cdot d_{r,s}(m,i)&=-(m+r)q\cdot d_{r,s}(m,i)+(r(i+q)-m(s+q))d_{r,s}(0,0),
		\end{align*}
		so
		\begin{align}\label{d_rs(m_i)-in-terms-of-d_rs(0_0)}
		(r-m)q\cdot d_{r,s}(m,i)&=(r(i+q)-m(s+q))d_{r,s}(0,0).
		\end{align}
		Choosing $m=r\ne 0$ and $i\ne s$ in \cref{d_rs(m_i)-in-terms-of-d_rs(0_0)} we get $d_{r,s}(0,0)=0$. Therefore, 
		\begin{align}\label{d_rs(m_i)=0-for-m-ne-r}
			d_{r,s}(m,i)=0, \mbox{ if }m\ne r,
		\end{align}
		by \cref{d_rs(m_i)-in-terms-of-d_rs(0_0)}. Now, taking $m=-n$ and $j=-i$ in \cref{one-half-der-in-terms-of-d_rs}, we have
		\begin{align}\label{d_ts(0_0)-and-d_rs(-n_i)_d_rs(n_-i)}
		2nq\cdot d_{r,s}(0,0)&=(2nq+ns+ri-rq)d_{r,s}(-n,i)+(2nq+ns+ri+rq)d_{r,s}(n,-i).
		\end{align}
		Choosing, moreover, $n=-r\ne 0$ (in particular, $n\ne r$), we obtain from \cref{d_ts(0_0)-and-d_rs(-n_i)_d_rs(n_-i),d_rs(m_i)=0-for-m-ne-r} that $0=(-3q-s+i)d_{r,s}(r,i)$, whence 
		\begin{align}\label{d_rs(r_i)=0-for-i-ne-3q+s}
			d_{r,s}(r,i)=0, \mbox{ if }i\ne 3q+s.
		\end{align}
		In view of \cref{d_rs(m_i)=0-for-m-ne-r,d_rs(r_i)=0-for-i-ne-3q+s}, it remains to prove that $d_{r,s}(r,3q+s)=0$. Take $m=r$, $i=3q+s$ and $n\not\in\{0,r\}$ in \cref{one-half-der-in-terms-of-d_rs}. Then $0=(n(4q+2s)-2r(j+q))d_{r,s}(r,3q+s)$ by \cref{d_rs(m_i)=0-for-m-ne-r}, so choosing $j\ne \frac nr(2q+s)-q$, we get $d_{r,s}(r,3q+s)=0$.
	\end{proof}
	
	\begin{lem}\label{vf_0s=0-unless-s=q}
		Let $\vf$ be a $\frac 12$-derivation of $\B(q)$ and $s\ne 0$. If $s\ne q$, then $\vf_{0,s}=0$. Moreover, $d_{0,q}(m,i)=0$, unless $(m,i)=(0,-2q)$.  
	\end{lem}
	\begin{proof}
		Write \cref{one-half-der-in-terms-of-d_rs} with $r=0$:
		\begin{align}\label{one-half-der-in-terms-of-d_0s}
		2(n(i+q)-m(j+q))d_{0,s}(m+n,i+j)&=(n(i+s+q)-m(j+q))d_{0,s}(m,i)\notag\\
		&\quad+(n(i+q)-m(j+s+q))d_{0,s}(n,j).
		\end{align}
		Then putting $n=j=0$ and $m\ne 0$, we obtain 
		\begin{align}\label{d_0s(m_i)-in-terms-of-d_0s(0_0)}
			d_{0,s}(m,i)=\left(1+sq^{-1}\right)d_{0,s}(0,0),\mbox{ if }m\ne 0.
		\end{align}
		We will now prove that $d_{0,s}(0,0)=0$.
		
		\textit{Case 1.} $q\in\Z$. Take $m=-n\ne 0$ and $j=-i$ in \cref{one-half-der-in-terms-of-d_0s} and then apply \cref{d_0s(m_i)-in-terms-of-d_0s(0_0)}:
		\begin{align*}
		2q\cdot d_{0,s}(0,0)=(2q+s)(d_{0,s}(-n,i)+d_{0,s}(n,-i))=2(2q+s)\left(1+sq^{-1}\right)d_{0,s}(0,0).
		\end{align*}
		Assuming $d_{0,s}(0,0)\ne 0$, we obtain $q=(2q+s)\left(1+sq^{-1}\right)$, whence $q^2+3sq+s^2=0$. It follows that $q=\frac{-3\pm \sqrt 5}2 s$, which contradicts $q\in\Z$.
		
		\textit{Case 2.} $q\not\in\Z$. Observe that $1+sq^{-1}\ne 0$ in this case. Taking $i=j=0$ in \cref{one-half-der-in-terms-of-d_0s}, we have:
		\begin{align*}
		2(n-m)q\cdot d_{0,s}(m+n,0)&=(ns+(n-m)q)d_{0,s}(m,0)+((n-m)q-ms)d_{0,s}(n,0).
		\end{align*}
		Choose $m,n,m\pm n\ne 0$. Then thanks to \cref{d_0s(m_i)-in-terms-of-d_0s(0_0)} and $1+sq^{-1}\ne 0$:
		\begin{align*}
			2(n-m)q\cdot d_{0,s}(0,0)&=(ns+(n-m)q)d_{0,s}(0,0)+((n-m)q-ms)d_{0,s}(0,0).
		\end{align*}
		Assuming $d_{0,s}(0,0)\ne 0$, we come to $(n-m)s=0$. So, $n-m=0$, a contradiction. 
		
		Thus, $d_{0,s}(0,0)=0$, and hence by \cref{d_0s(m_i)-in-terms-of-d_0s(0_0)}
		\begin{align}\label{d_0s(m_i)=0-for-m-ne-0-or-m=i=0}
		d_{0,s}(m,i)=0,\mbox{ if }m\ne 0.
		\end{align}
		It remains to analyze $d_{0,s}(0,i)$. To this end, put $m=-n\ne 0$ and $j=0$ in \cref{one-half-der-in-terms-of-d_0s} and use \cref{d_0s(m_i)=0-for-m-ne-0-or-m=i=0}:
		\begin{align*}
		2(2q+i)d_{0,s}(0,i)&=(2q+i+s)(d_{0,s}(-n,i)+d_{0,s}(n,0))=0.
		\end{align*}
		Therefore, 
		\begin{align}\label{d_0s(0_i)=0-for-i-ne-2q}
			d_{0,s}(0,i)=0,\mbox{ if }i\ne -2q.
		\end{align}
		Combining \cref{d_0s(m_i)=0-for-m-ne-0-or-m=i=0,d_0s(0_i)=0-for-i-ne-2q}, we conclude that $d_{0,s}(m,i)=0$ for $(m,i)\ne (0,-2q)$.
		
		Assume now that $s\ne q$. Take $m=0$, $n\ne 0$ and $i=-2q$ in \cref{one-half-der-in-terms-of-d_0s}. We have
		\begin{align*}
		-2q\cdot d_{0,s}(n,j-2q)=(s-q)d_{0,s}(0,-2q)-q\cdot d_{0,s}(n,j),
		\end{align*}
		whence $(s-q)d_{0,s}(0,-2q)=0$ by \cref{d_0s(m_i)=0-for-m-ne-0-or-m=i=0}. Since $s\ne q$, then $d_{0,s}(0,-2q)=0$.
	\end{proof}

\begin{lem}\label{alpha-half-der}
	Let $q\in\Z$. Then the linear map $\af:\B(q)\to\B(q)$ such that 
	\begin{align}\label{af(L_mi)=0-or-L_0-q}
		\af(L_{m,i})=
		\begin{cases}
			0, & (m,i)\ne(0,-2q),\\
			L_{0,-q}, & (m,i)=(0,-2q),
		\end{cases}
	\end{align}
	is a $\frac 12$-derivation of $\B(q)$.
\end{lem}
\begin{proof}
	Observe that $\af=\af_{0,q}$. In view of \cref{conditions-on-d} we need to check \cref{one-half-der-in-terms-of-d_rs} for $(r,s)=(0,q)$ and
	\begin{align}\label{d_0q(m_i)=0-or-a}
		d_{0,q}(m,i)=
		\begin{cases}
		0, & (m,i)\ne(0,-2q),\\
		1, & (m,i)=(0,-2q).
		\end{cases}
	\end{align}
	
	\textit{Case 1.} $(m,i),(n,j),(m+n,i+j)\ne(0,-2q)$. Then both sides of \cref{one-half-der-in-terms-of-d_rs} are zero.
	
	\textit{Case 2.} $(m,i)=(0,-2q)$. Then \cref{one-half-der-in-terms-of-d_rs} becomes
	\begin{align*}
	-2nq\cdot d_{0,q}(n,j-2q)=-nq\cdot d_{0,q}(n,j).
	\end{align*}
	If $n=0$, then it is trivially satisfied, otherwise both sides are zero by \cref{d_0q(m_i)=0-or-a}.
	
	\textit{Case 3.} $(n,j)=(0,-2q)$. Then \cref{one-half-der-in-terms-of-d_rs} becomes
	\begin{align*}
	2mq\cdot d_{r,s}(m,i-2q)=mq\cdot d_{r,s}(m,i),
	\end{align*}
	so this case is similar to Case 2.
	
	\textit{Case 4.} $(m+n,i+j)=(0,-2q)$. Then \cref{one-half-der-in-terms-of-d_rs} becomes
	\begin{align*}
	0=nq\cdot d_{0,q}(-n,i)+nq\cdot d_{0,q}(n,-i-2q),
	\end{align*}
	and again this holds by \cref{d_0q(m_i)=0-or-a}.
\end{proof}

Summarizing the results of \cref{vf_0s=0-unless-s=q,alpha-half-der,d_rs-is-zero-for-r-ne-0}, we have the following.
\begin{cor}\label{d_rs-is-zero-for-rs-not-00-or-0q}
	Let $\vf$ be a $\frac 12$-derivation of $\B(q)$. If $q\not\in\Z$, then $\vf_{r,s}=0$ for all $(r,s)\ne (0,0)$. If $q\in\Z$, then $\vf_{r,s}=0$ for all $(r,s)\not\in\{(0,0),(0,q)\}$ and $\vf_{0,q}\in\gen\af$.
\end{cor}
	
\begin{lem}\label{d_00(m_i)-is-d_00(0_0)}
	Let $\vf$ be a $\frac 12$-derivation of $\B(q)$ satisfying \cref{Dl_rs(L_mi)=d_rs(m_i)L_m+r_i+s}. Then $\vf_{0,0}\in\gen\id$.
\end{lem}
\begin{proof}
	Write \cref{one-half-der-in-terms-of-d_rs} for $r=s=0$:
	\begin{align}\label{one-half-der-in-terms-of-d_00}
		2(n(i+q)-m(j+q))d_{0,0}(m+n,i+j)=(n(i+q)-m(j+q))(d_{0,0}(m,i)+d_{0,0}(n,j)).
	\end{align}
	Taking $n=j=0$ and $m\ne 0$ in \cref{one-half-der-in-terms-of-d_rs}, we obtain
	\begin{align*}
		2 d_{0,0}(m,i)=d_{0,0}(m,i)+d_{0,0}(0,0),
	\end{align*}
	whence
	\begin{align}\label{d_00(m_i)=d_00(0_0)-for-m-ne-0}
		d_{0,0}(m,i)=d_{0,0}(0,0), \mbox{ if }m\ne 0.
	\end{align}
	Now, put $m=-n\ne 0$ and $j=0$ in \cref{one-half-der-in-terms-of-d_00} and use \cref{d_00(m_i)=d_00(0_0)-for-m-ne-0}:
	\begin{align*}
		2(2q+i)d_{0,0}(0,i)=(2q+i)(d_{0,0}(-n,i)+d_{0,0}(n,0))=2(2q+i)d_{0,0}(0,0).
	\end{align*}
	Hence,
	\begin{align}\label{d_00(0_i)=d_00(0_0)-for-i-ne-2q}
	d_{0,0}(0,i)=d_{0,0}(0,0), \mbox{ if }i\ne -2q.
	\end{align}
	Finally, substitute $(m,i)=(0,-2q)$ in \cref{one-half-der-in-terms-of-d_00}. Then
	\begin{align*}
	2n\cdot d_{0,0}(n,j-2q)=n(d_{0,0}(0,-2q)+d_{0,0}(n,j)).
	\end{align*}
	So, for $n\ne 0$ we deduce from \cref{d_00(m_i)=d_00(0_0)-for-m-ne-0} that $2 d_{0,0}(0,0)=d_{0,0}(0,-2q)+d_{0,0}(0,0)$, whence $d_{0,0}(0,-2q)=d_{0,0}(0,0)$.
	Combining this with \cref{d_00(m_i)=d_00(0_0)-for-m-ne-0,d_00(0_i)=d_00(0_0)-for-i-ne-2q}, we conclude that $\vf_{0,0}=d_{0,0}(0,0)\id$.
\end{proof}

\begin{prop}\label{Dl(B(q))=<id_af>}
	Let $q\ne 0$. Then
	\begin{align*}
		\Dl(\B(q))=
		\begin{cases}
		\gen{\id}, & q\not\in\Z,\\
		\gen{\id,\af}, & q\in\Z\setminus\{0\}.
		\end{cases}
	\end{align*}
\end{prop}
\begin{proof}
	A consequence of \cref{d_rs-is-zero-for-rs-not-00-or-0q,alpha-half-der,d_00(m_i)-is-d_00(0_0)}.
\end{proof}

	\subsubsection{The case $q=0$}
	In this case \cref{one-half-der-in-terms-of-d_rs} reduces to
	\begin{align}\label{one-half-der-in-terms-of-d_rs-for-q=0}
	2(ni-mj)d_{r,s}(m+n,i+j)&=(n(i+s)-(m+r)j)d_{r,s}(m,i)\notag\\
	&\quad+((n+r)i-m(j+s))d_{r,s}(n,j).
	\end{align}
	
\begin{lem}\label{d_rs=0-for(r_s)-ne-(0_0)}
	Let $\vf$ be a $\frac 12$-derivation of $\B(0)$. If $(r,s)\ne (0,0)$, then $\vf_{r,s}=0$.
\end{lem}
\begin{proof}
	Put $m=j=0$ in \cref{one-half-der-in-terms-of-d_rs-for-q=0}:
	\begin{align}\label{d_rs(n_i)-in-terms-of-d_rs(0_i)-and-d_rs(n_0)}
		2ni\cdot d_{r,s}(n,i)&=n(i+s)d_{r,s}(0,i)+i(n+r)d_{r,s}(n,0).
	\end{align}
	In particular, substituting $n=0$, $i\ne 0$ and $n\ne 0$, $i=0$ in \cref{d_rs(n_i)-in-terms-of-d_rs(0_i)-and-d_rs(n_0)}, we see that
	\begin{align} \label{d_rs(0_0)=0-for-(r_s)-ne-(0_0)}
		d_{r,s}(0,0)=0,\mbox{ if }(r,s)\ne(0,0).
	\end{align}
	On the other hand, taking $m=-n$ and $j=0$ in \cref{one-half-der-in-terms-of-d_rs-for-q=0}, we have
	\begin{align}\label{2ni.d_rs(0_i)=...}
		2ni\cdot d_{r,s}(0,i)&=n(i+s)d_{r,s}(-n,i)+(n(i+s)+ri)d_{r,s}(n,0).
	\end{align}
	Observe that \cref{2ni.d_rs(0_i)=...} with $i=0$ together with \cref{d_rs(0_0)=0-for-(r_s)-ne-(0_0)} give
	\begin{align}\label{d_rs(-n_0)=-d_rs(n_0)}
	d_{r,s}(-n,0)=-d_{r,s}(n,0),\mbox{ if }s\ne 0.
	\end{align}
	Furthermore, if $i=-j$ and $m=0$, then \cref{one-half-der-in-terms-of-d_rs-for-q=0} becomes
	\begin{align}\label{2ni.d_rs(n_0)=...}
	2ni\cdot d_{r,s}(n,0)&=(n(i+s)+ri)d_{r,s}(0,i)+i(n+r)d_{r,s}(n,-i).
	\end{align}
	As above, this implies
	\begin{align}\label{d_rs(0_-i)=-d_rs(0_i)}
	d_{r,s}(0,-i)=-d_{r,s}(0,i),\mbox{ if }r\ne 0.
	\end{align}
	
	\textit{Case 1.} $r\ne 0$ and $s\ne 0$. 
	Then $n=-r$ in \cref{d_rs(n_i)-in-terms-of-d_rs(0_i)-and-d_rs(n_0),2ni.d_rs(n_0)=...} together with \cref{d_rs(-n_0)=-d_rs(n_0)} give
	\begin{align}
	2i\cdot d_{r,s}(-r,i)&=(i+s)d_{r,s}(0,i),\label{2i.d_rs(-r_i)=(i+s)d_rs(0_i)}\\
	2i\cdot d_{r,s}(r,0)&=-s\cdot d_{r,s}(0,i).\label{2i.d_rs(-r_0)=s.d_rs(0_i)}
	\end{align}
	On the other hand, $n=r$ in \cref{2ni.d_rs(0_i)=...} gives
	\begin{align}\label{2i.d_rs(0_i)=...}
	2i\cdot d_{r,s}(0,i)&=(i+s)d_{r,s}(-r,i)+(2i+s)d_{r,s}(r,0).
	\end{align}
	Multiplying \cref{2i.d_rs(0_i)=...} by $2i$ and using \cref{2i.d_rs(-r_i)=(i+s)d_rs(0_i),2i.d_rs(-r_0)=s.d_rs(0_i)}, we get
	\begin{align*}
		4i^2\cdot d_{r,s}(0,i)&=(i+s)^2d_{r,s}(0,i)-s(2i+s)d_{r,s}(0,i)=i^2\cdot d_{r,s}(0,i).
	\end{align*}
	Combining this with \cref{d_rs(0_0)=0-for-(r_s)-ne-(0_0)}, we conclude that
	\begin{align}\label{d_rs(0_i)=0}
		d_{r,s}(0,i)=0,\mbox{ if }r\ne 0\mbox{ and }s\ne 0.
	\end{align}
	Similarly, $i=-s$ in \cref{d_rs(n_i)-in-terms-of-d_rs(0_i)-and-d_rs(n_0),2ni.d_rs(0_i)=...} together with \cref{d_rs(0_-i)=-d_rs(0_i)} yield
	\begin{align}
		2n\cdot d_{r,s}(n,-s)&=(n+r)d_{r,s}(n,0),\label{2n.d_rs(n_-s)=(n+r)d_rs(n_0)}\\
		2n\cdot d_{r,s}(0,s)&=-r\cdot d_{r,s}(n,0).\label{2n.d_rs(0_s)=-r.d_rs(n_0)}
	\end{align}
	Now, $i=s$ in \cref{2ni.d_rs(n_0)=...} gives
	\begin{align}\label{2n.d_rs(n_0)=...}
	2n\cdot d_{r,s}(n,0)=(2n+r)d_{r,s}(0,s)+(n+r)d_{r,s}(n,-s).
	\end{align}
	Multiplying \cref{2n.d_rs(n_0)=...} by $2n$ and using \cref{2n.d_rs(n_-s)=(n+r)d_rs(n_0),2n.d_rs(0_s)=-r.d_rs(n_0)}, we come to
	\begin{align*}
	4n^2\cdot d_{r,s}(n,0)=-r(2n+r)d_{r,s}(n,0)+(n+r)^2d_{r,s}(n,0)=n^2d_{r,s}(n,0).
	\end{align*}
	So, together with \cref{d_rs(0_0)=0-for-(r_s)-ne-(0_0)} this results in
	\begin{align}\label{d_rs(n_0)=0}
		d_{r,s}(n,0)=0,\mbox{ if }r\ne 0\mbox{ and }s\ne 0.
	\end{align}
	It follows from \cref{d_rs(0_i)=0,d_rs(n_0)=0,d_rs(n_i)-in-terms-of-d_rs(0_i)-and-d_rs(n_0)} that
	\begin{align}\label{d_rs(n_i)=0}
		d_{r,s}(n,i)=0,\mbox{ if }r\ne 0, s\ne 0,n\ne 0\mbox{ and }i\ne 0.
	\end{align}
	Combining \cref{d_rs(n_i)=0,d_rs(n_0)=0,d_rs(0_i)=0}, we obtain 
	\begin{align}\label{d_rs(n_i)=0-if-r-ne-0-and-s-ne-0}
	d_{r,s}(n,i)=0,\mbox{ if }r\ne 0\mbox{ and }s\ne 0.
	\end{align}
	
	\textit{Case 2.} $r\ne 0$ and $s=0$. Substituting $n=-r$ into \cref{d_rs(n_i)-in-terms-of-d_rs(0_i)-and-d_rs(n_0),2ni.d_rs(0_i)=...}, we have
	\begin{align}
	2i\cdot d_{r,0}(-r,i)&=i\cdot d_{r,0}(0,i),\label{2i.d_r0(-r_i)=i.d_r0(0_i)}\\
	2i\cdot d_{r,0}(0,i)&=i\cdot d_{r,0}(r,i).\label{2i.d_r0(0_i)=i.d_r0(r_i)}
	\end{align}
	On the other hand, the substitution $n=r$ in \cref{d_rs(n_i)-in-terms-of-d_rs(0_i)-and-d_rs(n_0),2ni.d_rs(0_i)=...} leads to
	\begin{align}
	2i\cdot d_{r,0}(r,i)&=i\cdot d_{r,0}(0,i)+2i\cdot d_{r,0}(r,0),\label{2i.d_r0(r_i)=i.d_r0(0_i)+2i.d_r0(r_0)}\\
	2i\cdot d_{r,0}(0,i)&=i\cdot d_{r,0}(-r,i)+2i\cdot d_{r,0}(r,0).\label{2i.d_r0(0_i)=i.d_r0(-r_i)+2i.d_r0(r_0)}
	\end{align}
	It follows from \cref{2i.d_r0(-r_i)=i.d_r0(0_i),2i.d_r0(0_i)=i.d_r0(-r_i)+2i.d_r0(r_0)} that $3i\cdot d_{r,0}(0,i)=4i\cdot d_{r,0}(r,0)$. On the other hand, \cref{2i.d_r0(0_i)=i.d_r0(r_i),2i.d_r0(r_i)=i.d_r0(0_i)+2i.d_r0(r_0)} imply that $3i\cdot d_{r,0}(0,i)=2i\cdot d_{r,0}(r,0)$. This, together with \cref{d_rs(0_0)=0-for-(r_s)-ne-(0_0)} results in
	\begin{align}\label{d_r0(0_i)=0-for-r-ne-0}
		d_{r,0}(0,i)=0,\mbox{ if }r\ne 0.
	\end{align} 
	Hence, \cref{d_rs(n_i)-in-terms-of-d_rs(0_i)-and-d_rs(n_0),2ni.d_rs(0_i)=...,2ni.d_rs(n_0)=...} take the following form
	\begin{align}
	2ni\cdot d_{r,0}(n,i)&=i(n+r)d_{r,0}(n,0),\label{2ni.d_r0(n_i)=i(n+r)d_r0(n_0)}\\
	0&=ni\cdot d_{r,0}(-n,i)+i(n+r)d_{r,0}(n,0),\label{0=ni.d_r0(-n_i)+i(n+r)d_r0(n_0)}\\
	2ni\cdot d_{r,0}(n,0)&=i(n+r)d_{r,0}(n,-i).\label{2ni.d_r0(n_0)=i(n+r)d_r0(n_-i)}
	\end{align}
	Replacing $i$ by $-i$ in \cref{2ni.d_r0(n_0)=i(n+r)d_r0(n_-i)} and combining this with \cref{2ni.d_r0(n_i)=i(n+r)d_r0(n_0)}, we arrive at $i^2(r-n)(r+3n)d_{r,0}(n,0)=0$, so $d_{r,0}(n,0)=0$, whenever $n\not\in\{r,-\frac r3\}$. Hence, $d_{r,0}(n,i)=0$ for $n\not\in\{r,-\frac r3\}$ by \cref{2ni.d_r0(n_i)=i(n+r)d_r0(n_0),d_r0(0_i)=0-for-r-ne-0}. However, if $n\in\{r,-\frac r3\}$, then $-n\not\in\{r,-\frac r3\}$, so $d_{r,0}(-n,i)=0$, yielding $d_{r,0}(n,0)=0$ thanks to \cref{0=ni.d_r0(-n_i)+i(n+r)d_r0(n_0)}. Thus, 
	\begin{align}\label{d_r0(n_i)=0-if-r-ne-0}
		d_{r,0}(n,i)=0,\mbox{ if }r\ne 0,
	\end{align}
	completing the proof of this case.
	
	\textit{Case 3.} $r=0$ and $s\ne 0$. Taking $i=-s$ in \cref{d_rs(n_i)-in-terms-of-d_rs(0_i)-and-d_rs(n_0),2ni.d_rs(n_0)=...}, we obtain
	\begin{align}
		2n\cdot d_{0,s}(n,-s)&=n\cdot d_{0,s}(n,0),\label{2n.d_0s(n_-s)=n.d_0s(n_0)}\\
		2n\cdot d_{0,s}(n,0)&=n\cdot d_{0,s}(n,s).\label{2n.d_0s(n_0)=n.d_0s(n_s)}
	\end{align}
On the other hand, substituting $i=s$ in \cref{d_rs(n_i)-in-terms-of-d_rs(0_i)-and-d_rs(n_0),2ni.d_rs(n_0)=...}, we get
\begin{align}
	2n\cdot d_{0,s}(n,s)&=2n\cdot d_{0,s}(0,s)+n\cdot d_{0,s}(n,0),\label{2n.d_0s(n_s)=2n.d_0s(0_s)+n.d_0s(n_0)}\\
	2n\cdot d_{0,s}(n,0)&=2n\cdot d_{0,s}(0,s)+n\cdot d_{0,s}(n,-s).\label{2n.d_0s(n_0)=2n.d_0s(0_s)+n.d_0s(n_-s)}
\end{align}
Observe that \cref{2n.d_0s(n_-s)=n.d_0s(n_0),2n.d_0s(n_0)=2n.d_0s(0_s)+n.d_0s(n_-s)} imply that $3n\cdot d_{0,s}(n,0)=4n\cdot d_{0,s}(0,s)$, while \cref{2n.d_0s(n_0)=n.d_0s(n_s),2n.d_0s(n_s)=2n.d_0s(0_s)+n.d_0s(n_0)} yield $3n\cdot d_{0,s}(n,0)=2n\cdot d_{0,s}(0,s)$. So, taking into account \cref{d_rs(0_0)=0-for-(r_s)-ne-(0_0)} as well, we conclude that
\begin{align}\label{d_0s(n_0)=0-if-s-ne-0}
	d_{0,s}(n,0)=0,\mbox{ if }s\ne 0.
\end{align}
Thus, \cref{d_rs(n_i)-in-terms-of-d_rs(0_i)-and-d_rs(n_0),2ni.d_rs(0_i)=...,2ni.d_rs(n_0)=...} reduce to
\begin{align}
	2ni\cdot d_{0,s}(n,i)&=n(i+s)d_{0,s}(0,i),\label{2ni.d_0s(n_i)=n(i+s)d_0s(0_i)}\\
	2ni\cdot d_{0,s}(0,i)&=n(i+s)d_{0,s}(-n,i),\label{2ni.d_0s(0_i)=n(i+s)d_0s(-n_i)}\\
	0&=n(i+s)d_{0,s}(0,i)+ni\cdot d_{0,s}(n,-i).\label{0=n(i+s)d_0s(0_i)+ni.d_0s(n_-i)}
\end{align}
Replacing $n$ by $-n$ in \cref{2ni.d_0s(0_i)=n(i+s)d_0s(-n_i)} and combining it with \cref{2ni.d_0s(n_i)=n(i+s)d_0s(0_i)}, we come to $n(s-i)(s+3i)d_{0,s}(0,i)=0$, whence $d_{0,s}(0,i)=0$ for $i\not\in\{s,-\frac s3\}$. Consequently, $d_{0,s}(n,i)=0$, if $i\not\in\{s,-\frac s3\}$ in view of \cref{d_0s(n_0)=0-if-s-ne-0,2ni.d_0s(n_i)=n(i+s)d_0s(0_i)}. Finally, if $i\in\{s,-\frac s3\}$, then $-i\not\in\{s,-\frac s3\}$, so $d_{0,s}(n,-i)=0$. Hence $d_{0,s}(0,i)=0$ from \cref{0=n(i+s)d_0s(0_i)+ni.d_0s(n_-i)}, and we again have $d_{0,s}(n,i)=0$ thanks to \cref{2ni.d_0s(n_i)=n(i+s)d_0s(0_i)}. Resuming, we have proved
\begin{align}\label{d_0s(n_i)=0-if-s-ne-0}
	d_{0,s}(n,i)=0,\mbox{ if }s\ne 0.
\end{align}

The result now follows from \cref{d_rs(n_i)=0-if-r-ne-0-and-s-ne-0,d_r0(n_i)=0-if-r-ne-0,d_0s(n_i)=0-if-s-ne-0}.
\end{proof}

\begin{lem}\label{d_00(m_i)=d_00(m'_i')}
	Let $\vf$ be a $\frac 12$-derivation of $\B(0)$ satisfying \cref{Dl_rs(L_mi)=d_rs(m_i)L_m+r_i+s}. Then $d_{0,0}(m,i)=d_{0,0}(m',i')$ for all $(m,i),(m',i')\in\Z\times\Z\setminus\{(0,0)\}$.
\end{lem}
\begin{proof}
	Equality \cref{one-half-der-in-terms-of-d_rs-for-q=0} written for $r=s=0$ becomes
	\begin{align*}
	2(ni-mj)d_{0,0}(m+n,i+j)=(ni-mj)(d_{0,0}(m,i)+d_{0,0}(n,j)).
	\end{align*}
	Hence, 
	\begin{align}\label{2d_00(m+n_i+j)=d_00(m_i)+d_00(n_j)}
		2d_{0,0}(m+n,i+j)=d_{0,0}(m,i)+d_{0,0}(n,j),\text{ if }ni-mj\ne 0.
	\end{align}
	In particular, taking $m=j=0$ in \cref{2d_00(m+n_i+j)=d_00(m_i)+d_00(n_j)}, we obtain
	\begin{align}\label{2d_00(n_i)=d_00(0_i)+d_00(n_0)}
	2d_{0,0}(n,i)=d_{0,0}(n,0)+d_{0,0}(0,i),\text{ if }ni\ne 0.
	\end{align}
	On the other hand, taking $m=-n$ and $j=0$ in \cref{2d_00(m+n_i+j)=d_00(m_i)+d_00(n_j)}, we get
	\begin{align}\label{2d_00(0_i)=d_00(-n_i)+d_00(n_0)}
	2d_{0,0}(0,i)=d_{0,0}(n,0)+d_{0,0}(-n,i),\text{ if }ni\ne 0.
	\end{align}
	Similarly, the substitution $j=-i$ and $m=0$ into \cref{2d_00(m+n_i+j)=d_00(m_i)+d_00(n_j)} leads to
	\begin{align}\label{2d_00(n_0)=d_00(0_i)+d_00(n_-i)}
	2d_{0,0}(n,0)=d_{0,0}(0,i)+d_{0,0}(n,-i),\text{ if }ni\ne 0.
	\end{align}
	Now, combining \cref{2d_00(n_i)=d_00(0_i)+d_00(n_0),2d_00(n_0)=d_00(0_i)+d_00(n_-i)} we get
	\begin{align*}
		3d_{0,0}(n,0)=2d_{0,0}(0,i)+d_{0,0}(0,-i),\text{ if }ni\ne 0.
	\end{align*}
	In particular, $d_{0,0}(n,0)=d_{0,0}(-n,0)$ for all $n$. Similarly, \cref{2d_00(n_i)=d_00(0_i)+d_00(n_0),2d_00(0_i)=d_00(-n_i)+d_00(n_0)} yield
	\begin{align*}
		3d_{0,0}(0,i)=2d_{0,0}(n,0)+d_{0,0}(-n,0)=3d_{0,0}(n,0),
	\end{align*}
	i.e.
	\begin{align}\label{d_00(n_0)=d_00(0_i)}
		d_{0,0}(n,0)=d_{0,0}(0,i),\text{ if }ni\ne 0.
	\end{align}
	Then \cref{d_00(n_0)=d_00(0_i),2d_00(n_i)=d_00(0_i)+d_00(n_0)} imply
	\begin{align*}
		d_{0,0}(n,i)=d_{0,0}(n,0)=d_{0,0}(0,i),
	\end{align*}
	which is clearly a constant not depending on $(n,i)$ such that $ni\ne 0$. 
\end{proof}

\begin{prop}\label{Dl(B(0))=<id_af>}
	We have $\Dl(\B(0))=\gen{\id,\af}$, where $\af$ is as in \cref{alpha-half-der}.
\end{prop}
\begin{proof}
	It follows from \cref{d_00(m_i)=d_00(m'_i'),d_rs=0-for(r_s)-ne-(0_0)} that $\Dl(\B(0))\sst\gen{\id,\af}$. Conversely, any element of $\gen{\id,\af}$ is a $\frac 12$-derivation of $\B(0)$ by \cref{alpha-half-der}.
\end{proof}

We may thus join \cref{Dl(B(0))=<id_af>,Dl(B(q))=<id_af>} to get the following theorem.

\begin{thrm}\label{Dl(B(q))-full-description}
	For all $q\in\C$ we have
	\begin{align*}
	\Dl(\B(q))=
	\begin{cases}
	\gen{\id}, & q\not\in\Z,\\
	\gen{\id,\af}, & q\in\Z.
	\end{cases}
	\end{align*}
\end{thrm}

Filippov proved that each nonzero $\delta$-derivation ($\delta\neq0,1$) of a Lie algebra, 
 gives a non-trivial ${\rm Hom}$-Lie algebra structure \cite[Theorem 1]{fil1}.
 Hence, by Theorem \ref{Dl(B(q))-full-description}, we have the following corollary.

 \begin{cor}
The Lie algebra $\B(q)_{q\in\Z}$ admits a non-trivial ${\rm Hom}$-Lie algebra structure.
 \end{cor}

\subsection{Transposed Poisson structures on $\B(q)$}\label{tpa-alg}

Using \cref{Dl(B(q))-full-description} we can describe all the transposed Poisson structures on $(\B(q),[\cdot,\cdot])$.
\begin{thrm}\label{thalg}
	If $q\not\in\Z$, then all the transposed Poisson structures on $(\B(q),[\cdot,\cdot])$ are trivial. If $q\in\Z$, then, up to an isomorphism, there is only one non-trivial transposed Poisson structure $\cdot$ on $(\B(q),[\cdot,\cdot])$ given by
	\begin{align}\label{L_0-2q.L_0-2q=L_0-q}
		L_{0,-2q}\cdot L_{0,-2q}=L_{0,-q}.
	\end{align}
\end{thrm}
\begin{proof}
	Let $(\B(q),\cdot,[\cdot,\cdot])$ be a transposed Poisson algebra, so that $(\B(q),\cdot)$ is commutative and \cref{Trans-Leibniz} holds. For any $(m,i)\in\Z\times\Z$ denote by $\vf^{m,i}$ the left multiplication by $L_{m,i}$ in $(\B(q),\cdot)$, so that
	\begin{align}\label{L_mi.L_nj=vf^mi(L_nj)}
	L_{m,i}\cdot L_{n,j}=\vf^{m,i}(L_{n,j}).
	\end{align}
	Since $(\B(q),\cdot)$ is commutative, we also have
	\begin{align}\label{L_mi.L_nj=vf^nj(L_mi)}
	L_{m,i}\cdot L_{n,j}=L_{n,j}\cdot L_{m,i}=\vf^{n,j}(L_{m,i}).
	\end{align}
	
	By \cref{Trans-Leibniz} we have $\vf^{m,i}\in\Dl(\B(q))$. If $q\not\in\Z$, then it follows from \cref{Dl(B(q))-full-description} that $\vf^{m,i}=a^{m,i}\id$ for some $a^{m,i}\in\C$. It is then an immediate consequence of \cref{L_mi.L_nj=vf^mi(L_nj),L_mi.L_nj=vf^nj(L_mi)} that $a^{m,i}=0$ for all $(m,i)\in\Z\times\Z$. Thus, the product $\cdot$ is trivial for $q\not\in\Z$.
	
	So, we will focus on the case $q\in\Z$. Write, using \cref{Dl(B(q))-full-description}, 
	\begin{align}\label{vf^mi=a^mi.id+b^mi.af}
	\vf^{m,i}=a^{m,i}\id+b^{m,i}\af,
	\end{align}
	where $a^{m,i},b^{m,i}\in\C$ and $\af$ is given by \cref{af(L_mi)=0-or-L_0-q}. On the one hand, by \cref{L_mi.L_nj=vf^mi(L_nj),af(L_mi)=0-or-L_0-q,vf^mi=a^mi.id+b^mi.af},
	\begin{align}\label{L_mi.L_nj=a^mi.L_0-2q+b^mi.L_0-q}
	L_{m,i}\cdot L_{n,j}=a^{m,i}L_{n,j}+b^{m,i}\af(L_{n,j})=
	\begin{cases}
	a^{m,i}L_{n,j}, & (n,j)\ne(0,-2q),\\
	a^{m,i}L_{0,-2q}+b^{m,i}L_{0,-q}, & (n,j)=(0,-2q).
	\end{cases}
	\end{align}
	On the other hand, by \cref{L_mi.L_nj=vf^nj(L_mi),af(L_mi)=0-or-L_0-q,vf^mi=a^mi.id+b^mi.af},
	\begin{align}\label{L_mi.L_nj=a^nj.L_0-2q+b^nj.L_0-q}
	L_{m,i}\cdot L_{n,j}=a^{n,j}L_{m,i}+b^{n,j}\af(L_{m,i})=
	\begin{cases}
	a^{n,j}L_{m,i}, & (m,i)\ne(0,-2q),\\
	a^{n,j}L_{0,-2q}+b^{n,j}L_{0,-q}, & (m,i)=(0,-2q).
	\end{cases}
	\end{align}
	
	\textit{Case 1.} $(m,i),(n,j)\ne(0,-2q)$. Then $a^{m,i}L_{n,j}=a^{n,j}L_{m,i}$ by \cref{L_mi.L_nj=a^mi.L_0-2q+b^mi.L_0-q,L_mi.L_nj=a^nj.L_0-2q+b^nj.L_0-q}, so taking $(m,i)\ne(n,j)$ we conclude that $a^{m,i}=a^{n,j}=0$. Thus, $L_{m,i}\cdot L_{n,j}=0$.
	
	\textit{Case 2.} $(m,i)=(0,-2q)$, $(n,j)\ne(0,-2q)$. Then $a^{0,-2q}L_{n,j}=a^{n,j}L_{0,-2q}+b^{n,j}L_{0,-q}$ by \cref{L_mi.L_nj=a^mi.L_0-2q+b^mi.L_0-q,L_mi.L_nj=a^nj.L_0-2q+b^nj.L_0-q}. Choosing $(n,j)\ne(0,-q)$, we obtain $a^{0,-2q}=0$, so $L_{m,i}\cdot L_{n,j}=0$.
	
	\textit{Case 3.} $(m,i)\ne(0,-2q)$, $(n,j)=(0,-2q)$. This case is symmetric to Case 2, so again $L_{m,i}\cdot L_{n,j}=0$.
	
	\textit{Case 4.} $(m,i)=(n,j)=(0,-2q)$. Then $L_{m,i}\cdot L_{n,j}=a^{0,-2q}L_{0,-2q}+b^{0,-2q}L_{0,-q}=b^{0,-2q}L_{0,-q}$, because $a^{0,-2q}=0$, as proved in Case 2.
	
	Thus, the product in $(\B(q),\cdot)$ is of the form
	\begin{align}\label{L_0-2q.L_0-2q=cL_0-q}
		L_{0,-2q}\cdot L_{0,-2q}=cL_{0,-q},
	\end{align}
	where $c\in\C$. Assume that $c\ne 0$ (otherwise the transposed Poisson structure is trivial). Observe that $L_{0,-q}\in Z(\B(q))$, where $Z(\B(q))$ is the center of the Lie algebra $(\B(q),[\cdot,\cdot])$. Indeed,
	\begin{align*}
		[L_{m,i}, L_{0,-q}] = (0\cdot (i + q) - m(-q + q))L_{m+n,i+j}=0
	\end{align*}
	for all $(m,i)\in\Z\times\Z$. Hence the linear map $\phi$ such that $\phi(L_{m,i})=L_{m,i}$ for $(m,i)\ne(0,-q)$ and $\phi(L_{0,-q})=kL_{0,-q}$ is an automorphism of $(\B(q),[\cdot,\cdot])$ for any $k\in\C^*$. If $q\ne 0$, then taking $k=c^{-1}$, we obtain an isomorphic transposed Poisson structure $*$ on $(\B(q),[\cdot,\cdot])$ in which the only non-zero product is 
	\begin{align*}
	    L_{0,-2q}*L_{0,-2q}=\phi(L_{0,-2q})*\phi(L_{0,-2q})=\phi(L_{0,-2q}\cdot L_{0,-2q})=\phi(cL_{0,-q})=c\cdot c^{-1} L_{0,-q}=L_{0,-q},
	\end{align*}
	so, up to an isomorphism, we may consider $c=1$ in \cref{L_0-2q.L_0-2q=cL_0-q}. If $q=0$, then taking $k=c$, we obtain an isomorphic transposed Poisson structure $*$ on $(\B(q),[\cdot,\cdot])$ in which the only non-zero product is 
	\begin{align*}
	    L_{0,0}*L_{0,0}=c^{-1}\phi(L_{0,0})*c^{-1}\phi(L_{0,0})=c^{-2}\phi(L_{0,0}\cdot L_{0,0})=c^{-2}\phi(cL_{0,0})=c^{-2}\cdot c^2 L_{0,0}=L_{0,0},
	\end{align*}
	so again, up to an isomorphism, it suffices to take $c=1$ in \cref{L_0-2q.L_0-2q=cL_0-q}.
	
	Conversely, consider the commutative algebra structure on $\B(q)$ given by \cref{L_0-2q.L_0-2q=L_0-q}. It is clearly associative. To prove \cref{Trans-Leibniz}, observe that $\B(q)\cdot\B(q)\sst Z(\B(q))$. Hence, the right-hand side of \cref{Trans-Leibniz} is always zero. In fact, the left-hand side of \cref{Trans-Leibniz} is zero as well, because $[\B(q),\B(q)]\sst\Ann(\B(q))$, where $\Ann(\B(q))$ is the annihilator of $(\B(q),\cdot)$. For, assuming $[L_{m,i}, L_{n,j}]\in\gen{L_{0,-2q}}$ we obtain from \cref{[L_mi_L_nj]=(n(i+q)-m(j+q))L_m+n_i+j} that $m+n=0$ and $i+j=-2q$. But then
	\begin{align*}
		n(i+q)-m(j+q)=n(i+q)+n(-i-2q+q)=n(i+q)-n(i+q)=0,
	\end{align*}
	so $[L_{m,i}, L_{n,j}]=0$. Thus, $[L_{m,i}, L_{n,j}]\in\Ann(\B(q))$ for all $(m,i),(n,j)\in\Z\times\Z$, as needed.
\end{proof}

	\section{Transposed Poisson structures on Block Lie superalgebras}
	
	We are going to study the Lie superalgebra $\S(q)$ whose definition comes from \cite{xia16} with the only difference that $q$ will be an arbitrary complex number and the set of indices of the basis elements will be the whole $\Z\times\Z$. Thus, $\S(q)_0=\B(q)$ with basis $\{L_{m,i} \mid m, i \in \Z\}$ and multiplication given by \cref{[L_mi_L_nj]=(n(i+q)-m(j+q))L_m+n_i+j}. Now, $\S(q)_1$ is spanned by $\{G_{m,i} \mid m, i \in \Z\}$, where
	\begin{align}
		[L_{m,i},G_{n,j}]&=\left(n(i+q)-m\left(j+\frac q2\right)\right)G_{m+n,i+j},\label{[L_mi_G_nj]}\\
		[G_{m,i},G_{n,j}]&=2qL_{m+n,i+j}.\label{[G_mi_G_nj]}
	\end{align}
	
	\subsection{Even $\frac 12$-derivations of $\S(q)$}
	
	
	In this subsection we consider only \textit{even} linear maps $\vf:\S(q)\to\S(q)$, i.e. those which satisfy $\vf(\S(q)_i)\sst \S(q)_i$ for $i\in\{0,1\}$. We thus have $|\vf|=0$, so $\vf$ is a $\frac 12$-superderivation of $\S(q)$ if and only if $\vf$ is a usual $\frac 12$-derivation of $\S(q)$. 
	We now write 
	\begin{align*}
	\vf=\sum_{r,s\in\Z}\vf_{r,s},
	\end{align*}
	where 
	\begin{align}
	\vf_{r,s}(L_{m,i})=d^0_{r,s}(m,i)L_{m+r,i+s},\label{vf_rs(L_mi)=d^0_rs(m_i)L_m+r_i+s}\\
	\vf_{r,s}(G_{m,i})=d^1_{r,s}(m,i)G_{m+r,i+s}\label{vf_rs(G_mi)=d^1_rs(m_i)G_m+r_i+s}
	\end{align}
	for some $d^i_{r,s}(m,i)\in\C$, $i=0,1$. Since 
	\begin{align*}
		[(\S(q)_k)_{m,i},(\S(q)_l)_{n,j}]\sst(\S(q)_{k+l})_{m+n,i+j}
	\end{align*}
	for all $k,l\in\Z_2$ and $m,n,i,j\in\Z$ by \cref{[L_mi_G_nj],[G_mi_G_nj]}, then $\vf\in\Dl^0(\S(q))$ if and only if $\vf_{r,s}\in\Dl^0(\S(q))$ for all $r,s\in\Z$.
	
	\begin{lem}\label{conditions-on-d^0-and-d^1}
		Let $\vf_{r,s}:\S(q)\to\S(q)$ be a linear map satisfying \cref{vf_rs(L_mi)=d^0_rs(m_i)L_m+r_i+s,vf_rs(G_mi)=d^1_rs(m_i)G_m+r_i+s}. 
		Then $\vf_{r,s}\in\Dl^0(\S(q))$ if and only if $\vf_{r,s}|_{\S(q)_0}\in\Dl(\B(q))$ and
		\begin{align}
		2\left(n(i+q)-m\left(j+\frac q2\right)\right)d^1_{r,s}(m+n,i+j)&=\left(n(i+s+q)-(m+r)\left(j+\frac q2\right)\right)d^0_{r,s}(m,i)\notag\\
		&\quad+\left((n+r)(i+q)-m\left(j+s+\frac q2\right)\right)d^1_{r,s}(n,j),\label{one-half-der-in-terms-of-d^0_rs-and-d^1_rs}\\
		2qd^0_{r,s}(m+n,i+j)&=q(d^1_{r,s}(m,i)+d^1_{r,s}(n,j)).\label{2qd^0_rs(m+n_i+j)=q(d^1_rs(m_i)+d^1_rs(n_j))}
		\end{align}
	\end{lem}
\begin{proof}
	By \cref{[L_mi_G_nj]} we see that $\vf_{r,s}\in\Dl^0(\S(q))$ if and only if
	\begin{align}
		2\left(n(i+q)-m\left(j+\frac q2\right)\right)\vf_{r,s}(G_{m+n,i+j})=[\vf_{r,s}(L_{m,i}),G_{n,j}]+[L_{m,i},\vf_{r,s}(G_{n,j})],\label{2(n(i+q)-m(j+q-over-2))vf_rs(G_m+n_i+j)=[L_mi_vf_rs(G_nj)]}\\
		4q\vf_{r,s}(L_{m+n,i+j})=[\vf_{r,s}(G_{m,i}),G_{n,j}]+[G_{m,i},\vf_{r,s}(G_{n,j})].\label{4qvf_rs(L_m+n_i+j)=[vf_rs(G_mi)_G_nj]+[G_mi_vf_rs(G_nj)]}
	\end{align}
	In view of \cref{vf_rs(L_mi)=d^0_rs(m_i)L_m+r_i+s,vf_rs(G_mi)=d^1_rs(m_i)G_m+r_i+s} the left-hand side of \cref{2(n(i+q)-m(j+q-over-2))vf_rs(G_m+n_i+j)=[L_mi_vf_rs(G_nj)]} equals
	\begin{align*}
	2\left(n(i+q)-m\left(j+\frac q2\right)\right)d^1_{r,s}(m+n,i+j)G_{m+n+r,i+j+s},
	\end{align*}
	while the right-hand side of \cref{2(n(i+q)-m(j+q-over-2))vf_rs(G_m+n_i+j)=[L_mi_vf_rs(G_nj)]} is
	\begin{align*}
		&[d^0_{r,s}(m,i)L_{m+r,i+s},G_{n,j}]+[L_{m,i},d^1_{r,s}(n,j)G_{n+r,j+s}]\\
		&\quad=\left(n(i+s+q)-(m+r)\left(j+\frac q2\right)\right)d^0_{r,s}(m,i)G_{m+n+r,i+j+s}\\
		&\quad\quad+\left((n+r)(i+q)-m\left(j+s+\frac q2\right)\right)d^1_{r,s}(n,j)G_{m+n+r,i+j+s}.
	\end{align*}
	Thus, we come to \cref{one-half-der-in-terms-of-d^0_rs-and-d^1_rs}. Now, the left-hand side of \cref{4qvf_rs(L_m+n_i+j)=[vf_rs(G_mi)_G_nj]+[G_mi_vf_rs(G_nj)]} is
	\begin{align*}
		4qd^0_{r,s}(m+n,i+j)L_{m+n+r,i+j+s},
	\end{align*}
	while the right-hand side of \cref{4qvf_rs(L_m+n_i+j)=[vf_rs(G_mi)_G_nj]+[G_mi_vf_rs(G_nj)]} equals
	\begin{align*}
		&[d^1_{r,s}(m,i)G_{m+r,i+s},G_{n,j}]+[G_{m,i},d^1_{r,s}(n,j)G_{n+r,j+s}]\\
		&\quad=2qd^1_{r,s}(m,i)L_{m+n+r,i+j+s}+2qd^1_{r,s}(n,j)L_{m+n+r,i+j+s},
	\end{align*}
	whence \cref{2qd^0_rs(m+n_i+j)=q(d^1_rs(m_i)+d^1_rs(n_j))}.
\end{proof}

We are going to specify the result of \cref{conditions-on-d^0-and-d^1} taking into account the description of $\Dl(\B(q))$ from \cref{Dl(B(q))-full-description}.

	\begin{lem}\label{conditions-on-d^0-and-d^1-specified}
	Let $\vf_{r,s}:\S(q)\to\S(q)$ be a linear map satisfying \cref{vf_rs(L_mi)=d^0_rs(m_i)L_m+r_i+s,vf_rs(G_mi)=d^1_rs(m_i)G_m+r_i+s} and such that $\vf_{r,s}|_{\S(q)_0}\in\Dl(\B(q))$. 
	
	If $q\not\in\Z$ or $q\in\Z$ and $(r,s)\not\in\{(0,0),(0,q)\}$, then $\vf_{r,s}\in\Dl^0(\S(q))$ if and only if
	\begin{align}
	2\left(n(i+q)-m\left(j+\frac q2\right)\right)d^1_{r,s}(m+n,i+j)&=\left((n+r)(i+q)-m\left(j+s+\frac q2\right)\right)d^1_{r,s}(n,j),\label{one-half-der-in-terms-of-d^1_rs}\\
	q(d^1_{r,s}(m,i)+d^1_{r,s}(n,j))&=0.\label{q(d^1_rs(m_i)+d^1_rs(n_j))=0}
	\end{align}
	
	If $q\in\Z\setminus\{0\}$, then $\vf_{0,0}\in\Dl^0(\S(q))$ if and only if
	\begin{align}
	\left(n(i+q)-m\left(j+\frac q2\right)\right)(2d^1_{0,0}(m+n,i+j)-d^1_{0,0}(n,j)-d^0_{0,0}(0,0))&=0,\label{one-half-der-in-terms-of-d^0_00-and-d^1_00}\\
	d^1_{0,0}(m,i)+d^1_{0,0}(n,j)=2d^0_{0,0}(0,0).\label{d^1_00(m_i)+d^1_00(n_j)=2d^0_00(m+n_i+j)}
	\end{align}
	
	If $q\in\Z\setminus\{0\}$, then $\vf_{0,q}\in\Dl^0(\S(q))$ if and only if
	\begin{align}
	2\left(n(i+q)-m\left(j+\frac q2\right)\right)d^1_{0,q}(m+n,i+j)&=\left(n(i+q)-m\left(j+\frac {3q}2\right)\right)d^1_{0,q}(n,j)\notag\\
	&\quad\text{for }(m,i)\ne(0,-2q),\notag\\
	2nd^1_{0,q}(n,j-2q)&=nd^1_{0,q}(n,j),\notag\\
	d^1_{0,q}(m,i)+d^1_{0,q}(n,j)&=0\text{ for }(m+n,i+j)\ne(0,-2q),\label{d^1_0q(m_i)+d^1_0q(n_j)=0}\\
	d^1_{0,q}(-n,-j-2q)+d^1_{0,q}(n,j)&=2d^0_{0,q}(0,-2q),\label{d^1_0q(-n_-j-2q)+d^1_0q(n_j)=2d^0_0q(0_-2q)}
	\end{align}
	
	If $q=0$, then $\vf_{0,0}\in\Dl^0(\S(q))$ if and only if
	\begin{align}\label{one-half-der-in-terms-of-d^0_00-and-d^1_00-q=0}
	(ni-mj)(2d^1_{0,0}(m+n,i+j)-d^0_{0,0}(m,i)-d^1_{0,0}(n,j))=0,
	\end{align}
	where $d^0_{0,0}(m,i)=d^0_{0,0}(m',i')$ for all $(m,i),(m',i')\in\Z\times\Z\setminus\{(0,0)\}$.
\end{lem}

Now we will consider each case from \cref{conditions-on-d^0-and-d^1-specified} separately.
\begin{lem}\label{vf_rs-for-(r_s)-ne-(0_0)-or-(0_q)}
	Let $\vf$ be a $\frac 12$-derivation of $\S(q)$. If $q\not\in\Z$ and $(r,s)\ne(0,0)$, or $q\in\Z$ and $(r,s)\not\in\{(0,0),(0,q)\}$, then $\vf_{r,s}=0$.
\end{lem}
\begin{proof}
	The map $\vf_{r,s}$ acts as in \cref{vf_rs(L_mi)=d^0_rs(m_i)L_m+r_i+s,vf_rs(G_mi)=d^1_rs(m_i)G_m+r_i+s}, and we know by \cref{Dl(B(q))-full-description} that $d^0_{r,s}=0$, so it remains to prove that $d^1_{r,s}=0$.
	
	\textit{Case 1.} $q\ne 0$. Then $d^1_{r,s}(m,i)+d^1_{r,s}(n,j)=0$ by \cref{q(d^1_rs(m_i)+d^1_rs(n_j))=0}. Taking $n=m$ and $j=i$, we get $d^1_{r,s}(m,i)=0$ for all $(m,i)\in\Z\times\Z$.
	
	\textit{Case 2.} $q=0$. Then 
	\begin{align}\label{2(ni-mj)d^1_rs(m+n_i+j)=((n+r)i-m(j+s))d^1_rs(n_j)}
		2(ni-mj)d^1_{r,s}(m+n,i+j)=((n+r)i-m(j+s))d^1_{r,s}(n,j)
	\end{align}
	by \cref{one-half-der-in-terms-of-d^1_rs}. Taking $m=j=0$ and $i\ne 0$ in \cref{2(ni-mj)d^1_rs(m+n_i+j)=((n+r)i-m(j+s))d^1_rs(n_j)}, we get
	\begin{align}\label{2nd^1_rs(n_i)=(n+r)d^1_rs(n_0)}
		2nd^1_{r,s}(n,i)=(n+r)d^1_{r,s}(n,0).
	\end{align}
	In particular, $d^1_{r,s}(n,i)$ does not depend on $i\ne 0$. On the other hand, $m=0$ and $j=-i\ne 0$ in \cref{2(ni-mj)d^1_rs(m+n_i+j)=((n+r)i-m(j+s))d^1_rs(n_j)} yield
	\begin{align}\label{2nd^1_rs(n_0)=(n+r)d^1_rs(n_i)}
		2nd^1_{r,s}(n,0)=(n+r)d^1_{r,s}(n,-i)=(n+r)d^1_{r,s}(n,i).
	\end{align}
	
	\textit{Case 2.1.} $n\not\in\{r,-\frac r3\}$. Then
	\begin{align*}
		\begin{vmatrix}
		2n & -(n+r)\\
		n+r & -2n
		\end{vmatrix}
		=(n+r)^2-4n^2=(r-n)(r+3n)\ne 0.
	\end{align*}
	Hence, the linear system \cref{2nd^1_rs(n_i)=(n+r)d^1_rs(n_0),2nd^1_rs(n_0)=(n+r)d^1_rs(n_i)} has only the trivial solution $d^1_{r,s}(n,i)=d^1_{r,s}(n,0)=0$.
	
	\textit{Case 2.2.} $n=r$. It follows from \cref{2(ni-mj)d^1_rs(m+n_i+j)=((n+r)i-m(j+s))d^1_rs(n_j)} that
	\begin{align}\label{2(ri-mj)d^1_rs(m+r_i+j)=(2ri-m(j+s))d^1_rs(r_j)}
		2(ri-mj)d^1_{r,s}(m+r,i+j)=(2ri-m(j+s))d^1_{r,s}(r,j).
	\end{align}
	If $r\ne 0$, then choosing $m\not\in\{0,-\frac{4r}3\}$, $i\ne\frac{m(j+s)}{2r}$ in \cref{2(ri-mj)d^1_rs(m+r_i+j)=(2ri-m(j+s))d^1_rs(r_j)} and using the result of Case 2.1, we get $d^1_{r,s}(r,j)=0$. If $r=0$, then choosing $m\ne 0$ in \cref{2(ri-mj)d^1_rs(m+r_i+j)=(2ri-m(j+s))d^1_rs(r_j)} we similarly get $d^1_{0,s}(0,j)=0$ whenever $j\ne -s$. Finally, choosing $m=-n\ne 0$ and $j=-i-s$ in \cref{2(ni-mj)d^1_rs(m+n_i+j)=((n+r)i-m(j+s))d^1_rs(n_j)} we obtain $sd^1_{0,s}(0,-s)=0$. Since $(r,s)\ne(0,0)$ by assumption, then $s\ne 0$, so $d^1_{0,s}(0,-s)=0$.
	
	\textit{Case 2.3.} $n=-\frac r3$. It follows from \cref{2(ni-mj)d^1_rs(m+n_i+j)=((n+r)i-m(j+s))d^1_rs(n_j)} that
	\begin{align}\label{-2(ri-over-3+mj)d^1_rs(m-r-over-3_i+j)=(2ri-over-3-m(j+s))d^1_rs(-r-over-3_j)}
	-2\left(\frac {ri}3+mj\right)d^1_{r,s}\left(m-\frac r3,i+j\right)=\left(\frac{2ri}3-m(j+s)\right)d^1_{r,s}\left(-\frac r3,j\right).
	\end{align}
	We may assume that $r\ne 0$, since $r=0$ was considered in Case 2.2. Then choosing $m\not\in\{0,\frac{4r}3\}$, $i\ne\frac{3m(j+s)}{2r}$ in \cref{-2(ri-over-3+mj)d^1_rs(m-r-over-3_i+j)=(2ri-over-3-m(j+s))d^1_rs(-r-over-3_j)} and using the result of Case 2.1, we get $d^1_{r,s}\left(-\frac r3,j\right)=0$.
\end{proof}

\begin{lem}\label{vf_0q-for-q-in-Z^*}
	Let $\vf$ be a $\frac 12$-derivation of $\S(q)$. If $q\in\Z\setminus\{0\}$, then $\vf_{0,q}=0$.
\end{lem}
\begin{proof}
	Taking $(m,i)=(n,j)\ne (0,-q)$ in \cref{d^1_0q(m_i)+d^1_0q(n_j)=0}, we obtain $d^1_{0,q}(m,i)=0$ for $(m,i)\ne(0,-q)$. In particular, $d^1_{0,q}(0,0)=0$. But then substituting $(m,i)=(0,-q)$ and $(n,j)=(0,0)$ in the same \cref{d^1_0q(m_i)+d^1_0q(n_j)=0}, we come to $d^1_{0,q}(0,-q)=0$. Thus, $d^1_{0,q}=0$, which implies $d^0_{0,q}(0,-2q)=0$ by \cref{d^1_0q(-n_-j-2q)+d^1_0q(n_j)=2d^0_0q(0_-2q)}. Since we already know by \cref{Dl(B(q))=<id_af>} that $d^0_{0,q}(m,i)=0$ for all $(m,i)\ne(0,-2q)$, then $d^0_{0,q}=0$.
\end{proof}

\begin{lem}\label{d^1_00(m_i)=d^0_00(0_0)}
	Let $\vf$ be a $\frac 12$-derivation of $\S(q)$. If $q\ne 0$, then $d^1_{0,0}(m,i)=d^0_{0,0}(0,0)$ for all $m,i\in\Z$.
\end{lem}
\begin{proof}
	The substitution $m=i=n=j=0$ in \cref{d^1_00(m_i)+d^1_00(n_j)=2d^0_00(m+n_i+j)} gives $d^1_{0,0}(0,0)=d^0_{0,0}(0,0)$. Now, substituting only $n=j=0$ in \cref{d^1_00(m_i)+d^1_00(n_j)=2d^0_00(m+n_i+j)}, we obtain $d^1_{0,0}(m,i)=2d^0_{0,0}(0,0)-d^1_{0,0}(0,0)=d^0_{0,0}(0,0)$. Observe that \cref{one-half-der-in-terms-of-d^0_00-and-d^1_00} is then trivially satisfied.
\end{proof}

\begin{lem}\label{d^1_00(m_i)=d^0_00(m'_i')}
	Let $\vf$ be a $\frac 12$-derivation of $\S(0)$. Then $d^1_{0,0}(m,i)=d^0_{0,0}(m',i')$ for all $(m,i),(m',i')\in\Z\times\Z\setminus\{(0,0)\}$.
\end{lem}
\begin{proof}
	Substituting $m=j=0$ and $n,i\ne 0$ in \cref{one-half-der-in-terms-of-d^0_00-and-d^1_00-q=0} we get
	\begin{align}\label{2d^1_00(n_i)-d^0_00(0_i)-d^1_00(n_0)=0}
	2d^1_{0,0}(n,i)-d^0_{0,0}(0,i)-d^1_{0,0}(n,0)=0,\mbox{ if }ni\ne 0.
	\end{align}
On the other hand, the substitution $n=i=0$ and $m,j\ne 0$ in \cref{one-half-der-in-terms-of-d^0_00-and-d^1_00-q=0} gives
	\begin{align}\label{2d^1_00(m_j)-d^0_00(m_0)-d^1_00(0_j)=0}
	2d^1_{0,0}(m,j)-d^0_{0,0}(m,0)-d^1_{0,0}(0,j)=0,\mbox{ if }mj\ne 0.
	\end{align}
Since $d^0_{0,0}(0,i)=d^0_{0,0}(m,0)$ for all $m,i\ne 0$ by \cref{d_00(m_i)=d_00(m'_i')}, it follows from \cref{2d^1_00(n_i)-d^0_00(0_i)-d^1_00(n_0)=0,2d^1_00(m_j)-d^0_00(m_0)-d^1_00(0_j)=0} (with $m=n$ and $i=j$) that
\begin{align}\label{d^1_00(n_0)=d^1_00(0_i)}
	d^1_{0,0}(n,0)=d^1_{0,0}(0,i),\mbox{ if }ni\ne 0.
\end{align}
Now take $m=-n\ne 0$, $i\ne 0$ and $j=0$ in \cref{one-half-der-in-terms-of-d^0_00-and-d^1_00-q=0}:
\begin{align*}
2d^1_{0,0}(0,i)-d^0_{0,0}(-n,i)-d^1_{0,0}(n,0)=0.
\end{align*}
In view of \cref{d^1_00(n_0)=d^1_00(0_i)} we conclude that
\begin{align}\label{d^1_00(0_i)=d^1_00(n_0)=d^0_00(-n_i)}
d^1_{0,0}(0,i)=d^1_{0,0}(n,0)=d^0_{0,0}(-n,i).
\end{align}
Since $d^0_{0,0}(-n,i)=d^0_{0,0}(0,i)$ by \cref{d_00(m_i)=d_00(m'_i')}, it follows from \cref{d^1_00(0_i)=d^1_00(n_0)=d^0_00(-n_i),2d^1_00(n_i)-d^0_00(0_i)-d^1_00(n_0)=0} that $d^1_{0,0}(n,i)=d^0_{0,0}(0,i)$ for all $n,i\ne 0$. Combining this with \cref{d^1_00(n_0)=d^1_00(0_i),d^1_00(0_i)=d^1_00(n_0)=d^0_00(-n_i)}, we obtain the desired result.
\end{proof}

\begin{lem}\label{alpha-and-beta-half-der-S(0)}
	The linear maps $\af,\bt:\S(0)\to\S(0)$ such that 
	\begin{align}
		\af(L_{m,i})&=
		\begin{cases}
			0, & (m,i)\ne(0,0),\\
			L_{0,0}, & (m,i)=(0,0),
		\end{cases}\ 
		\af(G_{m,i})=0,\label{af(L_mi)=0-or-L_0-0}\\
		\bt(G_{m,i})&=
		\begin{cases}
			0, & (m,i)\ne(0,0),\\
			G_{0,0}, & (m,i)=(0,0),
		\end{cases}\ 
		\bt(L_{m,i})=0\label{bt(L_mi)=0-or-G_0-0}
	\end{align}
	are $\frac 12$-derivations of $\S(0)$.
\end{lem}
\begin{proof}
	Let us first prove that $\af\in\Dl^0(\S(0))$. We have $\af=\af_{0,0}$, where
	\begin{align}\label{d^0_00(m_i)=0-or-1}
		d^0_{0,0}(m,i)=
		\begin{cases}
		0, & (m,i)\ne(0,0),\\
		1, & (m,i)=(0,0)
		\end{cases}
	\end{align}
	and $d^1_{0,0}=0$. By \cref{alpha-half-der} we know that $\vf_{0,0}|_{\S(0)_0}\in\Dl(\B(0))$. It remains to check \cref{one-half-der-in-terms-of-d^0_00-and-d^1_00-q=0} which reduces to $(ni-mj)d^0_{0,0}(m,i)=0$. The latter is trivial by \cref{d^0_00(m_i)=0-or-1}.
	
	Now we will prove that $\bt\in\Dl^0(\S(q))$. Again, $\bt=\bt_{0,0}$, but now
	\begin{align}\label{d^1_00(m_i)=0-or-1}
		d^1_{0,0}(m,i)=
		\begin{cases}
		0, & (m,i)\ne(0,0),\\
		1, & (m,i)=(0,0)
		\end{cases}
	\end{align}
	and $d^0_{0,0}=0$. The equality \cref{one-half-der-in-terms-of-d^0_00-and-d^1_00-q=0} that we need to verify takes the form
	\begin{align}\label{(ni-mj)(2d^1_00(m+n_i+j)-d^1_00(n_j))=0}
	(ni-mj)(2d^1_{0,0}(m+n,i+j)-d^1_{0,0}(n,j))=0.    
	\end{align}
	Consider the following cases.
	
	\textit{Case 1.} $(n,j),(m+n,i+j)\ne(0,0)$. Then both sides of \cref{(ni-mj)(2d^1_00(m+n_i+j)-d^1_00(n_j))=0} are zero because $d^1_{0,0}(m+n,i+j)=d^1_{0,0}(n,j)=0$ by \cref{d^1_00(m_i)=0-or-1}.
	
	\textit{Case 2.} $(n,j)=(0,0)$ or $(m+n,i+j)=(0,0)$. Then $ni-mj=0$, so again both sides of \cref{(ni-mj)(2d^1_00(m+n_i+j)-d^1_00(n_j))=0} are zero.
\end{proof}

\begin{prop}\label{Dl(S(q))=<id_af_bt>}
	For all $q\in\C$ we have
	\begin{align*}
		\Dl^0(\S(q))=
		\begin{cases}
		\gen{\id}, & q\ne 0,\\
		\gen{\id,\af,\bt}, & q=0.
		\end{cases}
	\end{align*}
\end{prop}
\begin{proof}
    This follows from \cref{vf_rs-for-(r_s)-ne-(0_0)-or-(0_q),vf_0q-for-q-in-Z^*,d^1_00(m_i)=d^0_00(0_0),d^1_00(m_i)=d^0_00(m'_i'),alpha-and-beta-half-der-S(0)}.
\end{proof}

\subsection{Odd $\frac 12$-derivations of $\S(q)$}

Recall that a linear map $\vf:\S(q)\to\S(q)$ is \textit{odd}, if $\vf(\S(q)_i)\sst \S(q)_{1-i}$ for $i\in\{0,1\}$. In this case $|\vf|=1$, and $\vf$ is a $\frac 12$-superderivation of $\S(q)$ if and only if 
 	\begin{align}
 		\varphi([x,y])&= \frac 12 \left([\varphi(x),y]+[x, \varphi(y)] \right),\ x\in\S(q)_0,\label{vf([x_y])-for-x-even}\\
 		\varphi([x,y])&= \frac 12 \left([\varphi(x),y]-[x, \varphi(y)] \right),\ x\in\S(q)_1.\label{vf([x_y])-for-x-odd}
 	\end{align}
Denote by $\Dl^1(\S(q))$ the space of odd $\frac 12$-superderivations of $\S(q)$. As usual, for any $\vf\in \Dl^1(\S(q))$, we write 
\begin{align*}
\vf=\sum_{r,s\in\Z}\vf_{r,s},
\end{align*}
where 
\begin{align}
\vf_{r,s}(L_{m,i})=d^0_{r,s}(m,i)G_{m+r,i+s},\label{vf_rs(L_mi)=d^0_rs(m_i)G_m+r_i+s}\\
\vf_{r,s}(G_{m,i})=d^1_{r,s}(m,i)L_{m+r,i+s}\label{vf_rs(G_mi)=d^1_rs(m_i)L_m+r_i+s}
\end{align}
for some $d^i_{r,s}(m,i)\in\C$, $i=0,1$. We have $\vf\in\Dl^1(\S(q))$ if and only if $\vf_{r,s}\in\Dl^1(\S(q))$ for all $r,s\in\Z$.

\begin{lem}\label{d^0-and-d^1-in-Delta^1_half}
	Let $\vf_{r,s}:\S(q)\to\S(q)$ be a linear map satisfying \cref{vf_rs(L_mi)=d^0_rs(m_i)G_m+r_i+s,vf_rs(G_mi)=d^1_rs(m_i)L_m+r_i+s}. 
	Then $\vf_{r,s}\in\Dl^0(\S(q))$ if and only if
	\begin{align}
	2(n(i + q) - m(j + q))d^0_{r,s}(m+n,i+j)&= d^0_{r,s}(m,i)\left(n\left(i+s+\frac q2\right)-(m+r)(j+q)\right)\notag\\
	&\quad+d^0_{r,s}(n,j)\left((n+r)(i+q)-m\left(j+s+\frac q2\right)\right),\label{vf_rs[L_mi_L_nj]-Dl^1-general}\\
	2\left(n(i+q)-m\left(j+\frac q2\right)\right)d^1_{r,s}(m+n,i+j)&= 2q\cdot d^0_{r,s}(m,i)\notag\\
	&\quad+((n+r)(i+q)-m(j+s+q))d^1_{r,s}(n,j),\label{vf_rs[L_mi_G_nj]-Dl^1-general}\\	
	4q\cdot d^0_{r,s}(m+n,i+j)&=\left(n(i+s+q)-(m+r)\left(j+\frac q2\right)\right)d^1_{r,s}(m,i)\notag\\
	&\quad+\left(m(j+s+q)-(n+r)\left(i+\frac q2\right)\right)d^1_{r,s}(n,j).\label{vf_rsG_mi_G_nj]-Dl^1-general}
	\end{align}
\end{lem}
\begin{proof}
	Fix $r,s\in\Z$. Writing \cref{vf([x_y])-for-x-even} for $\vf=\vf_{r,s}$ with $x=L_{m,i}$ and $y=L_{n,j}$, we have by \cref{vf_rs(L_mi)=d^0_rs(m_i)G_m+r_i+s}
	\begin{align*}
	2\vf_{r,s}([L_{m,i},L_{n,j}])&=[\vf_{r,s}(L_{m,i}),L_{n,j}]+[L_{m,i}, \vf_{r,s}(L_{n,j})]\\
	&=d^0_{r,s}(m,i)[G_{m+r,i+s},L_{n,j}]+d^0_{r,s}(n,j)[L_{m,i}, G_{n+r,j+s}].
	\end{align*}
	Then thanks to \cref{[L_mi_L_nj]=(n(i+q)-m(j+q))L_m+n_i+j,[L_mi_G_nj],vf_rs(L_mi)=d^0_rs(m_i)G_m+r_i+s} we get \cref{vf_rs[L_mi_L_nj]-Dl^1-general}. Similarly, $\vf=\vf_{r,s}$, $x=L_{m,i}$ and $y=G_{n,j}$ in \cref{vf([x_y])-for-x-even} result in
	\begin{align*}
	2\vf_{r,s}([L_{m,i},G_{n,j}])&= [\vf_{r,s}(L_{m,i}),G_{n,j}]+[L_{m,i}, \vf_{r,s}(G_{n,j})]\\
	&=d^0_{r,s}(m,i)[G_{m+r,i+s},G_{n,j}]+d^1_{r,s}(n,j)[L_{m,i},L_{n+r,j+s}].
	\end{align*}
	So, using \cref{[L_mi_L_nj]=(n(i+q)-m(j+q))L_m+n_i+j,[L_mi_G_nj],vf_rs(L_mi)=d^0_rs(m_i)G_m+r_i+s} we come to \cref{vf_rs[L_mi_G_nj]-Dl^1-general}. Now take $\vf=\vf_{r,s}$, $x=G_{m,i}$ and $y=G_{n,j}$ in \cref{vf([x_y])-for-x-odd} and use \cref{vf_rs(G_mi)=d^1_rs(m_i)L_m+r_i+s}:
	\begin{align*}
	2\varphi_{r,s}([G_{m,i},G_{n,j}])&= [\varphi_{r,s}(G_{m,i}),G_{n,j}]-[G_{m,i}, \varphi_{r,s}(G_{n,j})]\\	
	&=d^1_{r,s}(m,i)[L_{m+r,i+s},G_{n,j}]-d^1_{r,s}(n,j)[G_{m,i}, L_{n+r,j+s}].	
	\end{align*}
	Applying \cref{[L_mi_G_nj],[G_mi_G_nj],vf_rs(L_mi)=d^0_rs(m_i)G_m+r_i+s}, we arrive at \cref{vf_rsG_mi_G_nj]-Dl^1-general}.
\end{proof}

Let $\vf_{r,s}\in\Dl^1(\S(q))$ and $\psi_{m,i}$ be the left multiplication by $G_{m,i}$ in $\S(q)$. Then $\psi_{m,i}$ is an odd superderivation of $\S(q)$, so that the supercommutator $[\vf_{r,s},\psi_{m,i}]=\vf_{r,s}\circ\psi_{m,i}+\psi_{m,i}\circ\vf_{r,s}$ is an even $\frac 12$-superderivation of $\S(q)$, whose description was given in \cref{Dl(S(q))=<id_af_bt>}. So, on the one hand,
  \begin{align}
	&[\vf_{r,s},\psi_{m,i}](L_{n,j})=\vf_{r,s}([G_{m,i},L_{n,j}])+[G_{m,i},\vf_{r,s}(L_{n,j})]\notag\\
	&=\left(n\left(i+\frac q2\right)-m(j+q)\right)\vf_{r,s}(G_{m+n,i+j})+[G_{m,i},d^0_{r,s}(n,j)G_{n+r,j+s}]\notag\\
	&=\left(n\left(i+\frac q2\right)-m(j+q)\right)d^1_{r,s}(m+n,i+j)L_{m+n+r,i+j+s}+2q\cdot d^0_{r,s}(n,j)L_{m+n+r,i+j+s}\notag\\
	&=\left(\left(n\left(i+\frac q2\right)-m(j+q)\right)d^1_{r,s}(m+n,i+j)+2q\cdot d^0_{r,s}(n,j)\right)L_{m+n+r,i+j+s},\label{[vf_rs_psi_mi](L_nj)-in-terms-of-d}\\
	&[\vf_{r,s},\psi_{m,i}](G_{n,j})=\vf_{r,s}([G_{m,i},G_{n,j}])+[G_{m,i},\vf_{r,s}(G_{n,j})]\notag\\
	&=2q\cdot\vf_{r,s}(L_{m+n,i+j})+[G_{m,i},d^1_{r,s}(n,j)L_{n+r,j+s}]\notag\\
	&=2q\cdot d^0_{r,s}(m+n,i+j)G_{m+n+r,i+j+s} -\left(m(j+s+q)-(n+r)\left(i+\frac q2\right)\right)d^1_{r,s}(n,j)G_{m+n+r,i+j+s}\notag\\
	&=\left(2q\cdot d^0_{r,s}(m+n,i+j)-\left(m(j+s+q)-(n+r)\left(i+\frac q2\right)\right)d^1_{r,s}(n,j)\right)G_{m+n+r,i+j+s}.\label{[vf_rs_psi_mi](G_nj)-in-terms-of-d}
\end{align} 
And on the other hand, if $q\ne 0$, then
\begin{align}
[\vf_{r,s},\psi_{m,i}](L_{n,j})&=
\begin{cases}
0, & (r+m,s+i)\ne (0,0),\\
cL_{n,j}, & (r+m,s+i)=(0,0),
\end{cases}\label{[vf_rs_psi_mi](L_nj)-multiple-L_nj-q-ne-0}\\
[\vf_{r,s},\psi_{m,i}](G_{n,j})&=
\begin{cases}
0, & (r+m,s+i)\ne (0,0),\\
cG_{n,j}, & (r+m,s+i)=(0,0),
\end{cases}\label{[vf_rs_psi_mi](G_nj)-multiple-G_nj-q-ne-0}
\end{align}
for some constant $c\in\C$. And if $q=0$, then
\begin{align}
[\vf_{r,s},\psi_{m,i}](L_{n,j})&=
\begin{cases}
0, & (r+m,s+i)\ne (0,0),\\
c_1L_{n,j}, & (r+m,s+i)=(0,0)\text{ and }(n,j)\ne (0,0),\\
c_2L_{n,j}, & (r+m,s+i)=(0,0)\text{ and }(n,j)=(0,0),
\end{cases}\label{[vf_rs_psi_mi](L_nj)-multiple-L_nj-q=0}\\
[\vf_{r,s},\psi_{m,i}](G_{n,j})&=
\begin{cases}
0, & (r+m,s+i)\ne (0,0),\\
c_1G_{n,j}, & (r+m,s+i)=(0,0)\text{ and }(n,j)\ne (0,0),\\
c_3G_{n,j}, & (r+m,s+i)=(0,0)\text{ and }(n,j)=(0,0),
\end{cases}\label{[vf_rs_psi_mi](G_nj)-multiple-G_nj-q=0}
\end{align}
for some constants $c_1,c_2,c_3\in\C$.

\begin{lem}\label{vf_rs-in-Dl^1_half-for-q-ne-0}
	Let $q\ne 0$ and $\vf\in\Dl^1(\S(q))$. If $q\not\in 2\Z$, then $\vf_{r,s}=0$. Otherwise, $\vf_{r,s}=0$ for all $(r,s)\ne\left(0,\frac q2\right)$ and $\vf_{0,\frac q2}(L_{m,i})=0$ for all $m,i\in\Z$,
	$\vf_{0,\frac q2}(G_{m,i})=0$ for all $(m,i)\ne\left(0,-\frac{3q}2\right)$.
\end{lem}
\begin{proof}
	By \cref{[vf_rs_psi_mi](L_nj)-in-terms-of-d,[vf_rs_psi_mi](L_nj)-multiple-L_nj-q-ne-0,[vf_rs_psi_mi](G_nj)-in-terms-of-d,[vf_rs_psi_mi](G_nj)-multiple-G_nj-q-ne-0} we have for all $(m,i)\ne (-r,-s)$
	\begin{align}
		\left(n\left(i+\frac q2\right)-m(j+q)\right)d^1_{r,s}(m+n,i+j)+2q\cdot d^0_{r,s}(n,j)&=0,\label{(n(i+q-over-2)-m(j+q))d^1_rs(m+n_i+j)+2qd^0_rs(n_j)=0}\\
		2q\cdot d^0_{r,s}(m+n,i+j)-\left(m(j+s+q)-(n+r)\left(i+\frac q2\right)\right)d^1_{r,s}(n,j)&=0.\label{2qd^0_rs(m+n_i+j)-(m(j+s+q)-(n+r)(i+ q-over-2))d^1_rs(n_j)=0}
	\end{align}
Taking $m=n=0$ and $i\ne -s$ in \cref{(n(i+q-over-2)-m(j+q))d^1_rs(m+n_i+j)+2qd^0_rs(n_j)=0}, we obtain 
\begin{align*}
	d^0_{r,s}(0,j)=0\text{ for all }(r,s)\text{ and for all }j.
\end{align*}
Then $n=0$, $m\ne 0$ and $j\ne -q$ in \cref{(n(i+q-over-2)-m(j+q))d^1_rs(m+n_i+j)+2qd^0_rs(n_j)=0} give $d^1_{r,s}(m,i+j)=0$. Since any $k\in\Z$ can be written as $i+j$ with $i\ne -s$ and $j\ne -q$ (by choosing $j\not\in\{-q,k+s\}$), we obtain
\begin{align*}
	d^1_{r,s}(m,k)=0\text{ for all }(r,s)\text{ and for all }(m,k)\text{ with }m\ne 0.
\end{align*} 
Hence, $m=0$, $n\ne 0$ and $i\ne -s$ in \cref{(n(i+q-over-2)-m(j+q))d^1_rs(m+n_i+j)+2qd^0_rs(n_j)=0} yield 
\begin{align*}
	d^0_{r,s}(n,j)=0\text{ for all }(r,s)\text{ and for all }(n,j)\text{ with }n\ne 0.
\end{align*}
Thus, 
\begin{align*}
	d^0_{r,s}=0\text{ for all }(r,s).
\end{align*}
Now substitute $m=-n\ne 0$ and $i\ne -s$ into \cref{(n(i+q-over-2)-m(j+q))d^1_rs(m+n_i+j)+2qd^0_rs(n_j)=0} to get $\left(i+j+\frac {3q}2\right)d^1_{r,s}(0,i+j)=0$, so 
\begin{align*}
	d^1_{r,s}(0,k)=0\text{ for all }(r,s)\text{ and for all }k\ne-\frac{3q}2.
\end{align*}
Using \cref{2qd^0_rs(m+n_i+j)-(m(j+s+q)-(n+r)(i+ q-over-2))d^1_rs(n_j)=0} with $n=0$ and $j=-\frac{3q}2$, we come to $\left(m(s-\frac q2)-r\left(i+\frac q2\right)\right)d^1_{r,s}\left(0,-\frac{3q}2\right)=0$. If $r\ne 0$, then choosing $i\not\in\{\frac mr(s-\frac q2)-\frac q2,-s\}$ we conclude that $d^1_{r,s}\left(0,-\frac{3q}2\right)=0$. Otherwise, choosing $m\ne 0$ (so that $m\ne -r$), we see that $d^1_{0,s}\left(0,-\frac{3q}2\right)=0$ unless $s=\frac q2$.
\end{proof}

\begin{lem}\label{vf_rs-in-Dl^1_half-for-q=0}
	Let $\vf\in\Dl^1(\S(0))$. Then $\vf_{r,s}=0$ for all $(r,s)\ne(0,0)$, $\vf_{0,0}(G_{m,i})=0$ for all $(m,i)\ne (0,0)$, $\vf_{0,0}(L_{m,i})=cG_{m,i}$ for all $(m,i)\ne (0,0)$ and some constant $c\in\C$.
\end{lem}
\begin{proof}
	By \cref{[vf_rs_psi_mi](L_nj)-in-terms-of-d,[vf_rs_psi_mi](L_nj)-multiple-L_nj-q=0,[vf_rs_psi_mi](G_nj)-in-terms-of-d,[vf_rs_psi_mi](G_nj)-multiple-G_nj-q=0} we have for all $(m,i)\ne (-r,-s)$
	\begin{align}
	(ni-mj)d^1_{r,s}(m+n,i+j)&=0,\label{(ni-mj)d^1_rs(m+n_i+j)=0}\\
	(m(j+s)-i(n+r))d^1_{r,s}(n,j)&=0.\label{(m(j+s)-i(n+r))d^1_rs(n_j)=0}
	\end{align}
If $j\ne -s$, then choosing $m\not\in\{\frac{i(n+r)}{j+s},-r\}$ in \cref{(m(j+s)-i(n+r))d^1_rs(n_j)=0} we get $d^1_{r,s}(n,j)=0$. And if $n\ne -r$, then choosing $i\not\in\{\frac{m(j+s)}{n+r},-s\}$ in \cref{(m(j+s)-i(n+r))d^1_rs(n_j)=0} we again obtain $d^1_{r,s}(n,j)=0$. Thus, 
\begin{align*}
	d^1_{r,s}(n,j)=0\text{ for all }(n,j)\ne(-r,-s).
\end{align*}
Let $(r,s)\ne (0,0)$. Then take $m=-r-n$ and $j=-s-i$ in \cref{(ni-mj)d^1_rs(m+n_i+j)=0}. Since $ni-mj=ni-(r+n)(s+i)=-rs-ns-ri$, we have two cases. If $r\ne 0$, then choosing $i\not\in\{-\frac{ns}r-s,-s\}$, we conclude that $d^1_{r,s}(-r,-s)=0$. Otherwise, choosing $n\ne 0$ (so that $m=-n\ne -r$), we again come to $d^1_{r,s}(-r,-s)=0$. Thus, 
\begin{align*}
	d^1_{r,s}=0\text{ for all }(r,s)\ne(0,0).
\end{align*}
 
Regarding $d^0_{r,s}$, we see that it satisfies \cref{vf_rs[L_mi_L_nj]-Dl^1-general} with $q=0$, which has exactly the same form as \cref{one-half-der-in-terms-of-d_rs-for-q=0}. Applying the proofs of \cref{d_rs=0-for(r_s)-ne-(0_0),d_00(m_i)=d_00(m'_i')} to $d^0_{r,s}$, we conclude that 
\begin{align*}
	d^0_{r,s}&=0\text{ for all }(r,s)\ne(0,0),\\
	d^0_{0,0}(m,i)&=d^0_{0,0}(m',i')\text{ for all }(m,i),(m',i')\in\Z\times\Z\setminus\{(0,0)\}.
\end{align*}
\end{proof}

\begin{lem}\label{gamma-half-der-S(q)}
	Let $q\in 2\Z$. Then the linear map $\gm:\S(q)\to\S(q)$ such that 
	\begin{align}\label{gm(G_mi)=0-or-L_0-q}
	\gm(L_{m,i})=0,\ \gm(G_{m,i})&=
	\begin{cases}
	0, & (m,i)\ne\left(0,-\frac{3q}2\right),\\
	L_{0,-q}, & (m,i)=\left(0,-\frac{3q}2\right),
	\end{cases}
	\end{align}
	is an odd $\frac 12$-derivation of $\S(q)$.
\end{lem}
\begin{proof}
	Clearly, $\gm=\gm_{0,\frac q2}$. We are going to verify  \cref{vf_rs[L_mi_L_nj]-Dl^1-general,vf_rs[L_mi_G_nj]-Dl^1-general,vf_rsG_mi_G_nj]-Dl^1-general} for $(r,s)=\left(0,\frac q2\right)$ and
	\begin{align*}
	d^0_{0,\frac q2}&=0,\ d^1_{0,\frac q2}(m,i)=
	\begin{cases}
	0, & (m,i)\ne\left(0,-\frac{3q}2\right),\\
	1, & (m,i)=\left(0,-\frac{3q}2\right).
	\end{cases}
	\end{align*}
	Obviously, \cref{vf_rs[L_mi_L_nj]-Dl^1-general} is trivially satisfied.
	
	\textit{Case 1.} $(m,i),(n,j),(m+n,i+j)\ne\left(0,-\frac{3q}2\right)$. Then both sides of \cref{vf_rs[L_mi_G_nj]-Dl^1-general,vf_rsG_mi_G_nj]-Dl^1-general} are zero.
	
	\textit{Case 2.} $(m,i)=\left(0,-\frac{3q}2\right)$. Then \cref{vf_rs[L_mi_G_nj]-Dl^1-general,vf_rsG_mi_G_nj]-Dl^1-general} become
	\begin{align*}
	-qn\cdot d^1_{0,\frac q2}\left(n,j-\frac{3q}2\right)&= -\frac{qn}2\cdot d^1_{0,\frac q2}(n,j),\\	
	0&=qn\cdot d^1_{0,\frac q2}(n,j).
	\end{align*}
	Either $n=0$ or $d^1_{0,\frac q2}\left(n,j-\frac{3q}2\right)=d^1_{0,\frac q2}(n,j)=0$, so both sides of these equalities are zero.
	
	\textit{Case 3.} $(n,j)=\left(0,-\frac{3q}2\right)$. Then \cref{vf_rs[L_mi_G_nj]-Dl^1-general,vf_rsG_mi_G_nj]-Dl^1-general} become
	\begin{align*}
	2qm\cdot d^1_{0,\frac q2}\left(m,i-\frac{3q}2\right)&= 0,\\	
	0&=qm\cdot d^1_{0,\frac q2}(m,i).
	\end{align*}
	Either $m=0$ or $d^1_{0,\frac q2}\left(m,i-\frac{3q}2\right)=d^1_{0,\frac q2}(m,i)=0$, so again both sides are always zero.
	
	\textit{Case 4.} $(m+n,i+j)=\left(0,-\frac{3q}2\right)$. Then Then \cref{vf_rs[L_mi_G_nj]-Dl^1-general,vf_rsG_mi_G_nj]-Dl^1-general} become
	\begin{align*}
	0&= qn\cdot d^1_{0,\frac q2}(n,j),\\	
	0&=\frac{qn}2\cdot d^1_{0,\frac q2}(-n,i)-\frac{qn}2\cdot d^1_{0,\frac q2}(n,j).
	\end{align*}
	Either $n=0$ or $d^1_{0,\frac q2}(n,j)=d^1_{0,\frac q2}(-n,i)=0$, so once again both sides are zero.
\end{proof}

\begin{lem}\label{dl-and-eps-half-der-S(q)}
	The linear maps $\dl,\ve:\S(0)\to\S(0)$ such that 
	\begin{align}
	\dl(G_{m,i})&=0,\ \dl(L_{m,i})=
	\begin{cases}
	0, & (m,i)\ne(0,0),\\
	G_{0,0}, & (m,i)=(0,0),
	\end{cases}\label{dl(G_mi)=0-or-G_00}\\
	\ve(G_{m,i})&=0,\ \ve(L_{m,i})=G_{m,i},\label{ve(L_mi)=G_mi}
	\end{align}
	are odd $\frac 12$-derivations of $\S(0)$.
\end{lem}
\begin{proof}
	We first prove that $\dl\in\Dl^1(\S(0))$. Clearly, $\dl=\dl_{0,0}$, where
	\begin{align}\label{d^0_00-of-dl(m_i)=0-or-1}
	d^0_{0,0}(m,i)=
	\begin{cases}
	0, & (m,i)\ne(0,0),\\
	1, & (m,i)=(0,0)
	\end{cases}
	\end{align}
	and $d^1_{0,0}=0$. Equalities \cref{vf_rs[L_mi_G_nj]-Dl^1-general,vf_rsG_mi_G_nj]-Dl^1-general} are trivially satisfied, while \cref{vf_rs[L_mi_L_nj]-Dl^1-general} reduces to
	\begin{align}\label{2(ni-mj)d^0_00(m+n_i+j)=(ni-mj)(d^0_00(m_i)+d^0_00(n_j))}
			2(ni-mj)d^0_{0,0}(m+n,i+j)=(ni-mj)(d^0_{0,0}(m,i)+d^0_{0,0}(n,j)).
	\end{align}
	If $(0,0)\in\{(m,i),(n,j),(m+n,i+j)\}$, then $ni-mj=0$. Otherwise, $d^0_{0,0}(m+n,i+j)=d^0_{0,0}(m,i)=d^0_{0,0}(n,j)=0$ by \cref{d^0_00-of-dl(m_i)=0-or-1}. Thus, $\dl\in\Dl^1(\S(0))$ by \cref{d^0-and-d^1-in-Delta^1_half}.
	
	Now, let us prove that $\ve\in\Dl^1(\S(0))$. Again, $\ve=\ve_{0,0}$, but now
	\begin{align}\label{d^0_00-of-dl(m_i)=1}
	d^0_{0,0}(m,i)=1\text{ for all }(m,i)
	\end{align}
	and $d^1_{0,0}=0$. As above, \cref{vf_rs[L_mi_G_nj]-Dl^1-general,vf_rsG_mi_G_nj]-Dl^1-general} are trivial and \cref{vf_rs[L_mi_L_nj]-Dl^1-general} reduces to \cref{2(ni-mj)d^0_00(m+n_i+j)=(ni-mj)(d^0_00(m_i)+d^0_00(n_j))}. In view of \cref{d^0_00-of-dl(m_i)=1} the latter is also satisfied. So, $\ve\in\Dl^1(\S(0))$ by \cref{d^0-and-d^1-in-Delta^1_half}.
\end{proof}

\begin{prop}\label{Dl^1(S(q))=<gm_dl_ve>}
	For all $q\in\C$ we have
	\begin{align*}
	\Dl^1(\S(q))=
	\begin{cases}
	\{0\}, & q\not\in 2\Z,\\
	\gen{\gm}, & q\in 2\Z\setminus\{0\},\\
	\gen{\gm,\dl,\ve}, & q=0.
	\end{cases}
	\end{align*}
\end{prop}
\begin{proof}
	A consequence of \cref{vf_rs-in-Dl^1_half-for-q-ne-0,vf_rs-in-Dl^1_half-for-q=0,gamma-half-der-S(q),dl-and-eps-half-der-S(q)}.
\end{proof}

Observe that $\gm,\dl,\ve$ from \cref{Dl^1(S(q))=<gm_dl_ve>} are odd analogues of $\af,\bt,\id$ from \cref{Dl(S(q))=<id_af_bt>}.

\medskip 

 Filippov proved that each nonzero $\delta$-derivation ($\delta\neq0,1$) of a Lie algebra
 gives a non-trivial ${\rm Hom}$-Lie algebra structure \cite[Theorem 1]{fil1}.
 It is easy to see that a superanalog of his result is also true.
 Hence, by Proposition \ref{Dl^1(S(q))=<gm_dl_ve>} we have the following corollary.

 \begin{cor}
$\S(q)_{q\in2\Z}$ admits non-trivial ${\rm Hom}$-Lie superalgebra structures.
 \end{cor}

\subsection{Transposed Poisson superalgebra structures on $\S(q)$}
\label{tpa-salg}
With the help of \cref{Dl^1(S(q))=<gm_dl_ve>} we are now ready to describe the transposed Poisson superalgebra structures on $(\S(q),[\cdot,\cdot])$.
\begin{thrm}\label{tp-on-S(q)}
	If $q\ne 0$, then all the transposed Poisson superalgebra structures on $(\S(q),[\cdot,\cdot])$ are trivial. If $q=0$, then the non-trivial transposed Poisson superalgebra structures $(\S(q),\cdot,[\cdot,\cdot])$ on $(\S(q),[\cdot,\cdot])$ are, up to an isomorphism, of one of the following two forms
	\begin{align}
	L_{0,0}\cdot L_{0,0}&=L_{0,0},\ L_{0,0}\cdot G_{0,0}=G_{0,0}\cdot L_{0,0}=G_{0,0},\label{L_00^2-L_00G_00}\\
	L_{0,0}\cdot L_{0,0}&=L_{0,0}.\label{L_00^2=L_00}
	\end{align}
\end{thrm}
\begin{proof}
	Let $(\S(q),\cdot,[\cdot,\cdot])$ be a transposed Poisson superalgebra, i.e. $(\S(q),\cdot)$ is supercommutative and \cref{super-trans-leibniz} holds. Given $(m,i)\in\Z\times\Z$, denote by $\vf^{m,i}$ and $\psi^{m,i}$ the left multiplications by $L_{m,i}$ and $G_{m,i}$, respectively, in $(\S(q),\cdot)$, i.e.
	\begin{align}
	L_{m,i}\cdot L_{n,j}&=\vf^{m,i}(L_{n,j}),\ L_{m,i}\cdot G_{n,j}=\vf^{m,i}(G_{n,j}),\label{L_mi.L_nj=vf^mi(L_nj)-and-L_mi.G_nj=vf^mi(G_nj)}\\
	G_{m,i}\cdot L_{n,j}&=\psi^{m,i}(L_{n,j}),\ G_{m,i}\cdot G_{n,j}=\psi^{m,i}(G_{n,j})\label{G_mi.L_nj=psi^mi(L_nj)-and-G_mi.G_nj=psi^mi(G_nj)}
	\end{align}
	In view of supercommutativity of $(\S(q),\cdot)$, we have
	\begin{align}
	L_{m,i}\cdot L_{n,j}&=\vf^{n,j}(L_{m,i}),\ L_{m,i}\cdot G_{n,j}=\psi^{n,j}(L_{m,i}),\label{L_mi.L_nj=vf^nj(L_mi)-and-L_mi.G_nj=psi^nj(L_mi)}\\
	G_{m,i}\cdot L_{n,j}&=\vf^{n,j}(G_{m,i}),\ G_{m,i}\cdot G_{n,j}=-\psi^{n,j}(G_{m,i})\label{G_mi.L_nj=vf^nj(G_mi)-and-G_mi.G_nj=psi^nj(G_mi)}
	\end{align}
	
	By \cref{super-trans-leibniz} we have $\vf^{m,i}\in\Dl^0(\S(q))$ and $\psi^{m,i}\in\Dl^1(\S(q))$. 
	
	\textit{Case 1.} $q\not\in 2\Z$. Then $\vf^{m,i}=a^{m,i}\id$ for some $a^{m,i}\in\C$ by \cref{Dl(S(q))=<id_af_bt>} and $\psi^{m,i}=0$ by \cref{Dl^1(S(q))=<gm_dl_ve>}. It follows from \cref{L_mi.L_nj=vf^mi(L_nj)-and-L_mi.G_nj=vf^mi(G_nj),L_mi.L_nj=vf^nj(L_mi)-and-L_mi.G_nj=psi^nj(L_mi)} that $a^{m,i}=0$ for all $(m,i)\in\Z\times\Z$. So, $\cdot$ is trivial whenever $q\not\in 2\Z$.
	
	\textit{Case 2.} $q\in 2\Z\setminus\{0\}$. As above, $\vf^{m,i}=a^{m,i}\id$ by \cref{Dl(S(q))=<id_af_bt>}, so again  \cref{L_mi.L_nj=vf^mi(L_nj)-and-L_mi.G_nj=vf^mi(G_nj),L_mi.L_nj=vf^nj(L_mi)-and-L_mi.G_nj=psi^nj(L_mi)} imply $a^{m,i}=0$, whence $\vf^{m,i}=0$. So, $L_{m,i}\cdot L_{n,j}=L_{m,i}\cdot G_{n,j}=0$. However, 
	\begin{align}\label{psi^mi=a^mi.gm}
	\psi^{m,i}=b^{m,i}\gm
	\end{align}
	with $b^{m,i}\in\C$ and $\gm$ given by \cref{gm(G_mi)=0-or-L_0-q}. We immediately deduce from \cref{G_mi.L_nj=psi^mi(L_nj)-and-G_mi.G_nj=psi^mi(G_nj),gm(G_mi)=0-or-L_0-q,psi^mi=a^mi.gm} that $\psi^{m,i}(L_{n,j})=0$, so $G_{m,i}\cdot L_{n,j}=0$. Regarding $G_{m,i}\cdot G_{n,j}$, on the one hand, by \cref{G_mi.L_nj=psi^mi(L_nj)-and-G_mi.G_nj=psi^mi(G_nj),gm(G_mi)=0-or-L_0-q,psi^mi=a^mi.gm}
	\begin{align*}
	G_{m,i}\cdot G_{n,j}=b^{m,i}\gm(G_{n,j})=
	\begin{cases}
	0, & (n,j)\ne\left(0,-\frac{3q}2\right),\\
	b^{m,i}L_{0,-q}, & (n,j)=\left(0,-\frac{3q}2\right).
	\end{cases}
	\end{align*}
	On the other hand, by \cref{G_mi.L_nj=vf^nj(G_mi)-and-G_mi.G_nj=psi^nj(G_mi),gm(G_mi)=0-or-L_0-q,psi^mi=a^mi.gm}
	\begin{align*}
	G_{m,i}\cdot G_{n,j}=-b^{n,j}\gm(G_{m,i})=
	\begin{cases}
	0, & (m,i)\ne\left(0,-\frac{3q}2\right),\\
	-b^{n,j}L_{0,-q}, & (m,i)=\left(0,-\frac{3q}2\right).
	\end{cases}
	\end{align*}
	Thus, the product $G_{m,i}\cdot G_{n,j}$ is zero unless $(m,i)=(n,j)=\left(0,-\frac{3q}2\right)$. But $G_{0,-\frac{3q}2}\cdot G_{0,-\frac{3q}2}$ is also zero, because $(\S(q)_1,\cdot)$ is anticommutative. 
	
	
	
	\textit{Case 3.} $q=0$. Write $\vf^{(m,i)}=a^{m,i}\id+b^{m,i}\af+c^{m,i}\bt$ and $\psi^{(m,i)}=p^{m,i}\gm+q^{m,i}\dl+r^{m,i}\ve$ in view of \cref{Dl(S(q))=<id_af_bt>,Dl^1(S(q))=<gm_dl_ve>}. On the one hand, by \cref{L_mi.L_nj=vf^mi(L_nj)-and-L_mi.G_nj=vf^mi(G_nj),G_mi.L_nj=psi^mi(L_nj)-and-G_mi.G_nj=psi^mi(G_nj),af(L_mi)=0-or-L_0-0,bt(L_mi)=0-or-G_0-0,dl(G_mi)=0-or-G_00,ve(L_mi)=G_mi,gm(G_mi)=0-or-L_0-q}
	\begin{align*}
		L_{m,i}\cdot L_{n,j}&=a^{m,i}L_{n,j}+b^{m,i}\af(L_{n,j})+c^{m,i}\bt(L_{n,j})=
		\begin{cases}
		a^{m,i}L_{n,j}, & (n,j)\ne(0,0),\\
		\left(a^{m,i}+b^{m,i}\right)L_{0,0}, & (n,j)=(0,0).
		\end{cases}\\
		L_{m,i}\cdot G_{n,j}&=a^{m,i}G_{n,j}+b^{m,i}\af(G_{n,j})+c^{m,i}\bt(G_{n,j})=
		\begin{cases}
		a^{m,i}G_{n,j}, & (n,j)\ne(0,0),\\
		\left(a^{m,i}+c^{m,i}\right)G_{0,0}, & (n,j)=(0,0).
		\end{cases}\\
		G_{m,i}\cdot G_{n,j}&=p^{m,i}\gm(G_{n,j})+q^{m,i}\dl(G_{n,j})+r^{m,i}\ve(G_{n,j})=
		\begin{cases}
		0, & (n,j)\ne(0,0),\\
		p^{m,i}L_{0,0}, & (n,j)=(0,0).
		\end{cases}
	\end{align*}
	On the other hand, by \cref{L_mi.L_nj=vf^nj(L_mi)-and-L_mi.G_nj=psi^nj(L_mi),G_mi.L_nj=vf^nj(G_mi)-and-G_mi.G_nj=psi^nj(G_mi),af(L_mi)=0-or-L_0-0,bt(L_mi)=0-or-G_0-0,dl(G_mi)=0-or-G_00,ve(L_mi)=G_mi,gm(G_mi)=0-or-L_0-q}
	\begin{align*}
	L_{m,i}\cdot L_{n,j}&=a^{n,j}L_{m,i}+b^{n,j}\af(L_{m,i})+c^{n,j}\bt(L_{m,i})=
	\begin{cases}
	a^{n,j}L_{m,i}, & (m,i)\ne(0,0),\\
	\left(a^{n,j}+b^{n,j}\right)L_{0,0}, & (m,i)=(0,0).
	\end{cases}\\
	L_{m,i}\cdot G_{n,j}&=p^{n,j}\gm(L_{m,i})+q^{n,j}\dl(L_{m,i})+r^{n,j}\ve(L_{m,i})=
	\begin{cases}
	r^{n,j}G_{m,i}, & (m,i)\ne(0,0),\\
	\left(q^{n,j}+r^{n,j}\right)G_{0,0}, & (m,i)=(0,0).
	\end{cases}\\
	G_{m,i}\cdot G_{n,j}&=-p^{n,j}\gm(G_{m,i})-q^{n,j}\dl(G_{m,i})-r^{n,j}\ve(G_{m,i})=
	\begin{cases}
	0, & (m,i)\ne(0,0),\\
	-p^{n,j}L_{0,0}, & (m,i)=(0,0).
	\end{cases}
	\end{align*}
	
	Consider the product $L_{m,i}\cdot L_{n,j}$. If $(m,i),(n,j)\ne(0,0)$, then $a^{m,i}L_{n,j}=a^{n,j}L_{m,i}$, so taking $(m,i)\ne(n,j)$ we conclude that $a^{m,i}=a^{n,j}=0$. Thus, $L_{m,i}\cdot L_{n,j}=0$. If $(m,i)=(0,0)$, $(n,j)\ne(0,0)$, then $a^{0,0}L_{n,j}=\left(a^{n,j}+b^{n,j}\right)L_{0,0}=b^{n,j}L_{0,0}$. So, we obtain $a^{0,0}=b^{n,j}=0$, whence $L_{m,i}\cdot L_{n,j}=0$. Similarly, $(m,i)\ne(0,0)$, $(n,j)=(0,0)$ implies $L_{m,i}\cdot L_{n,j}=0$. Finally, if $(m,i)=(n,j)=(0,0)$, then $L_{m,i}\cdot L_{n,j}=\left(a^{0,0}+b^{0,0}\right)L_{0,0}=b^{0,0}L_{0,0}$, because $a^{0,0}=0$. 
	
	As to $L_{m,i}\cdot G_{n,j}$, the case $(m,i),(n,j)\ne(0,0)$ gives $L_{m,i}\cdot G_{n,j}=0$ and $r^{n,j}=0$, the case $(m,i)=(0,0)$, $(n,j)\ne(0,0)$ gives $L_{m,i}\cdot G_{n,j}=0$ and $q^{n,j}=0$, the case $(m,i)\ne(0,0)$, $(n,j)=(0,0)$ gives $L_{m,i}\cdot G_{n,j}=0$ and $r^{0,0}=c^{m,i}=0$, the case $(m,i)=(n,j)=(0,0)$ gives $L_{m,i}\cdot G_{n,j}=c^{0,0}G_{0,0}=q^{0,0}G_{0,0}$. 
	
	Regarding $G_{m,i}\cdot G_{n,j}$, we see that it is zero unless $(m,i)=(n,j)=(0,0)$, in which case it must also be zero thanks to anti-commutativity of $(\S(0)_1,\cdot)$. 
	
	Thus, the only possible non-zero products in $(\S(0),\cdot)$ are of the form 
	\begin{align*}
		L_{0,0}\cdot L_{0,0}=c_1L_{0,0},\ L_{0,0}\cdot G_{0,0}=G_{0,0}\cdot L_{0,0}=c_2G_{0,0}.
	\end{align*}
	Moreover, it follows from $L_{0,0}\cdot(L_{0,0}\cdot G_{0,0})=(L_{0,0}\cdot L_{0,0})\cdot G_{0,0}$ that $c_2^2=c_1c_2$. 
	
	Since $L_{0,0},G_{0,0}\in Z(\S(0))$, then any linear map of the form $\phi(L_{m,i})=L_{m,i}$ for $(m,i)\ne (0,0)$, $\phi(G_{n,j})=G_{n,j}$ for $(n,j)\ne (0,0)$, $\phi(L_{0,0})=k_{11}L_{0,0}+k_{12}G_{0,0}$ and $\phi(G_{0,0})=k_{21}L_{0,0}+k_{22}G_{0,0}$ with $k_{11}k_{22}\ne k_{12}k_{21}$ is an automorphism of $(\S(0),[\cdot,\cdot])$.

	Let $c_2\ne 0$. Then $c_1=c_2$, and we obtain the following multiplication table
	\begin{align}\label{L_00^2=c_1L_00_L_00G_00=c_1G_00}
	L_{0,0}\cdot L_{0,0}&=c_1L_{0,0},\ L_{0,0}\cdot G_{0,0}=G_{0,0}\cdot L_{0,0}=c_1G_{0,0}.
	\end{align}
	Applying $\phi$ with $k_{11}=k_{22}=c_1$ and $k_{12}=k_{21}=0$ to \cref{L_00^2=c_1L_00_L_00G_00=c_1G_00}, we come to a transposed Poisson structure isomorphic to \cref{L_00^2-L_00G_00}
	
	Let $c_2=0$. Then we obtain the following  multiplication table
	\begin{align}\label{L_00^2=c_1L_00}
	L_{0,0}\cdot L_{0,0}&=c_1L_{0,0}.
	\end{align}
	Applying $\phi$ with $k_{11}=c_1$, $k_{22}=1$ and $k_{12}=k_{21}=0$ to \cref{L_00^2=c_1L_00}, we come to a transposed Poisson structure isomorphic to \cref{L_00^2=L_00}.

	Conversely, each of the two associative and supercommutative multiplications \cref{L_00^2-L_00G_00,L_00^2=L_00} defines a transposed Poisson algebra structure on $\S(0)$, because $\S(0)\cdot\S(0)\sst\gen{L_{0,0},G_{0,0}}\sst Z(\S(0))$ and $[\S(0),\S(0)]\sst\Ann(\S(0))$. They are non-isomorphic because the dimensions of $S(0)^2$ are different under these two multiplications.
\end{proof}


\begin{thebibliography}{99}

 
  \bibitem{said2}
 Albuquerque H., Barreiro E., Benayadi S., Boucetta M., Sánchez J.M.,
 Poisson algebras and symmetric Leibniz bialgebra structures on oscillator Lie algebras,
 Journal of Geometry and Physics, 160 (2021), 103939.
 

 

\bibitem{bai20}
 Bai C., Bai R., Guo L., Wu Y.,
Transposed Poisson algebras, Novikov-Poisson algebras, and 3-Lie algebras, arXiv:2005.01110
 
 
 \bibitem{bfk22}
 Beites P. D., 
 Ferreira B. L. M., 
  Kaygorodov I.,  
 Transposed Poisson   structures,  arXiv:2207.00281



\bibitem{billig}
Billig Yu., 
    Towards Kac-van de Leur conjecture: locality of superconformal algebras, 
    Advances in Mathematics, 400 (2022),   108295. 

\bibitem{block58}
Block R., 
On torsion-free abelian groups and Lie algebras,
Proceedings of the American Mathematical Society, 9 (1958), 613--620.



\bibitem{cgz14}
  Chen H.,   Guo X., Zhao K., 
    Irreducible quasifinite modules over a class of Lie algebras of Block type, 
    Asian Journal of Mathematics, 18 (2014), 5, 817–827.
    
\bibitem{dz96}
    Đoković D., Zhao K., 
    Derivations, isomorphisms and second cohomology of generalized Block algebras,
    Algebra Colloquium, 3 (1996), 3, 245–272.

 \bibitem{FK21}
 Fehlberg Júnior R.,  Kaygorodov I.,
  On the Kantor product, II,
  Carpathian Mathematical Publications, 2021, to appear,  arXiv:2201.00174
 
 
 
\bibitem{FKL}
Ferreira B. L. M., Kaygorodov I., Lopatkin V., $\frac{1}{2}$-derivations of Lie algebras and transposed Poisson algebras, 
Revista de la Real Academia de Ciencias Exactas, Físicas y Naturales. Serie A. Matemáticas 115 (2021), 142. 
 
\bibitem{fil1} Filippov V.,
 $\delta$-Derivations of Lie algebras,
Siberian Mathematical Journal, 39 (1998), 6, 1218--1230.
 
  


  \bibitem{gwl21}
Guo X., Wang M., Liu X., 
$U(h)$-free modules over the Block algebra $\B(q)$,
Journal of Geometry and Physics, 169 (2021), 104333.
  
 
 \bibitem{jawo}
Jaworska-Pastuszak A., Pogorzały Z.,
 Poisson structures for canonical algebras,
 Journal of Geometry and Physics, 148 (2020), 103564.
   
 \bibitem{k12}
Kaygorodov I.,
$\delta$-superderivations of semisimple finite-dimensional Jordan superalgebras,
Mathematical Notes, 91 (2012),  1-2, 187--197.

 
	\bibitem{kk21}
Kaygorodov I., Khrypchenko M., 
    Poisson structures on finitary incidence algebras, 
    Journal of  Algebra, 578 (2021), 402--420.


 
\bibitem{hom}
Laraiedh I.,  Silvestrov S.,
    Transposed ${\rm Hom}$-Poisson and ${\rm Hom}$-pre-Lie Poisson algebras and bialgebras, arXiv:2106.03277
 
 
\bibitem{lgz18}
 Liu X., Guo X., Zhao K.,  
 Biderivations of the Block Lie algebras, 
 Linear Algebra and its Applications, 538 (2018), 43--55.
 

\bibitem{bihom}
 Ma T.,   Li B.,
Transposed ${\rm BiHom}$-Poisson algebras, arXiv:2201.00271

  
   

\bibitem{zel}
Martínez C., Zelmanov E., 
    Brackets, superalgebras and spectral gap, 
    São Paulo Journal of Mathematical Sciences, 13 (2019), 1, 112--132.
   
  
  
 

 
  \bibitem{oz99}
 Osborn   M.,  Zhao K., 
Infinite-dimensional Lie algebras of generalized Block type, 
Proceedings of the American Mathematical Society, 127 (1999), 6, 1641--1650.
  
 

\bibitem{suxixu12}
Su Y., Xia C., Xu Y., 
Quasifinite representations of a class of Block type Lie algebras $\B(q),$ 
Journal of Pure and Applied  Algebra, 216 (2012),  4, 923--934.

 
 
\bibitem{suxia18}
 Su Y., Xia C., Yuan L., 
    Classification of finite irreducible conformal modules over a class of Lie conformal algebras of Block type, Journal of Algebra, 499 (2018), 321–336.
    
\bibitem{suxia20}
Su Y., Xia C., Yuan L.,  
    Extensions of conformal modules over Lie conformal algebras of Block type, 
    Journal of Pure and Applied  Algebra, 224 (2020),  5, 106232, 24 pp.


\bibitem{suyue15} 
    Su Y., Yue X., Classification of $\mathbb Z_2$-graded modules of intermediate series over a Block type Lie algebra, Communications in Contemporary Mathematics, 17 (2015),  5, 1550059, 17 pp.
    
    

\bibitem{xia16}
    Xia C., Structure of two classes of Lie superalgebras of Block type, 
    International Journal of Mathematics, 27 (2016),  5, 1650038, 15 pp.
    
\bibitem{xia19}
 Xia C., 
 Classification of finite irreducible conformal modules over Lie conformal superalgebras of Block type, 
 Journal of Algebra, 531 (2019), 141--164.
 
\bibitem{xyz12}
Xia C.,  You T.,  Zhou L., 
    Structure of a class of Lie algebras of Block type, 
    Communications in  Algebra, 40 (2012),  8, 3113–3126.
 
   \bibitem{xu99}
Xu X., 
Generalizations of the Block algebras, 
Manuscripta mathematica, 100 (1999), 4, 489--518.


\bibitem{YYZ07}
Yao Y., Ye Y., Zhang P., 
Quiver Poisson algebras, 
Journal of  Algebra, 312 (2007), 2, 570--589.

    \bibitem{yh21}
Yuan L.,  Hua Q., 
$\frac{1}{2}$-(bi)derivations and transposed Poisson algebra structures on Lie algebras,
 Linear and Multilinear Algebra, 2021, DOI: 10.1080/03081087.2021.2003287

  
  
 
\bibitem{z10}
 Zusmanovich P., 
 On $\delta$-derivations of Lie algebras and superalgebras, 
 Journal of Algebra, 324 (2010), 12, 3470--3486.
 
 
 


 

   

   
\end{thebibliography}
\end{document}